\documentclass{amsart}
\usepackage{todonotes}
\usepackage{amsmath}
\usepackage{amsfonts}
\usepackage{amssymb}
\usepackage[all]{xy}

\usepackage{hyperref}
\usepackage{tikz}
\usetikzlibrary{patterns}

\theoremstyle{plain}
\newtheorem*{introtheorem}{Theorem}
\newtheorem{theorem}{Theorem}[section]
\newtheorem{proposition}[theorem]{Proposition}
\newtheorem{lemma}[theorem]{Lemma}
\newtheorem{corollary}[theorem]{Corollary}

\newtheorem*{proposition*}{Proposition}

\theoremstyle{definition}
\newtheorem{definition}[theorem]{Definition}
\newtheorem{example}[theorem]{Example}

\theoremstyle{remark}
\newtheorem{remark}[theorem]{Remark}

\newcommand{\secref}[1]{Section~\ref{#1}}
\newcommand{\thmref}[1]{Theorem~\ref{#1}}
\newcommand{\propref}[1]{Proposition~\ref{#1}}
\newcommand{\lemref}[1]{Lemma~\ref{#1}}
\newcommand{\corref}[1]{Corollary~\ref{#1}}

\newcommand{\exref}[1]{Example~\ref{#1}}

\newcommand{\remref}[1]{Remark~\ref{#1}}

\newcommand{\defref}[1]{Definition~\ref{#1}}
\newcommand{\figref}[1]{Figure~\ref{#1}}

\newcommand{\Z}{\mathbb{Z}}
\newcommand{\R}{\mathbb{R}}

\newcommand{\be}{\mathbf{e}}

\renewcommand{\bar}{\overline}

\begin{document}

\title[Face Group]{The Face Group of a Simplicial Complex}

\author[G.\ Lupton]{Gregory Lupton}
\author[N.\ Scoville]{Nicholas A. Scoville}
\author[P.C.\ Staecker]{P.~ Christopher Staecker}

\address{Department of Mathematics and Statistics, Cleveland State University, Cleveland OH 44115 U.S.A.}

\email{g.lupton@csuohio.edu}

\address{Department of Mathematics, Computer Science, and Statistics, Ursinus College, Collegeville PA 19426 U.S.A.}

\email{nscoville@ursinus.edu}

\address{Department of Mathematics, Fairfield University, Fairfield CT 06827 U.S.A.}

\email{cstaecker@fairfield.edu}

\date{\today}


\begin{abstract}  The edge group of a simplicial complex is a well-known,  combinatorial version of the fundamental group.  It is a group associated to a simplicial complex that consists of equivalence classes of edge loops and that is isomorphic to the ordinary (topological) fundamental group of the spatial realization. We define a counterpart to the edge group that likewise gives a combinatorial version of the second (higher) homotopy group.  Working entirely combinatorially, we show our group is an abelian group and also respects products.  We show that our combinatorially defined group is isomorphic to the ordinary (topological) second homotopy group of the spatial realization.
\end{abstract}

\maketitle

\section{Introduction}
The \emph{edge group} of a simplicial complex is a well-known combinatorial version of the fundamental group (see \cite{Mau96} for example---the construction is briefly summarized in \secref{sec: combinatorial face group} below). Let $E(X)$ denote the edge group of the simplicial complex $X$.  Then there is an isomorphism of groups $E(X) \cong \pi_1(|X|)$ between the (combinatorially defined) edge group of $X$ and the (topological) fundamental group of $|X|$, the spatial realization of $X$. In this paper, we give a like-minded combinatorial version of the second homotopy group of a simplicial complex.  

We start with a simplicial complex $X$ and construct a different group, which we call the \emph{face group} of $X$ and denote by  $F(X)$.  It is defined combinatorially, using the standard notion of \emph{contiguity} of simplicial maps together with an ingredient that effectively allows us to subdivide the domains of certain simplicial maps.  This second ingredient is the notion of (trivial) extensions of maps with domain a simplicial interval or a product of such. 

Once constructed, and its basic properties established, our main result about this group is the following.

\begin{introtheorem}
Let $X$ be a simplicial complex.  The face group of $X$ satisfies
$$F(X) \cong \pi_2(|X|),$$
where $|X|$ denotes the spatial realization of $X$.
\end{introtheorem}

We construct the face group and establish its basic properties  in \secref{sec: defn of Face Group} below, and show this isomorphism of groups in \thmref{thm: Face group iso homotopy group}.

Part of our development concerns simplicial approximation of continuous maps and homotopies (see \secref{sec: simp approx}), which usually involves barycentric subdivision of a simplicial complex.  But barycentric subdivision (of the domain of a simplicial map) seems to cause technical complications in the development we want here, to the point where its use is not feasible.    We avoid its use in our development as follows. First, we mimic in the combinatorial setting the definition of higher homotopy groups based on equivalence classes of maps from a cube modulo the boundary, rather than on equivalence classes of based maps from a sphere.  
Second, once we have rectangular domains for our simplicial maps, we are able to use the alternative to barycentric subdivision mentioned above (trivial extensions) to effect subdivisions of the domains of our representative simplicial maps.

The paper is organized as follows. In \secref{sec: simplicial basics} we review some basics about simplicial complexes and in \secref{sec: extn contiguity} we develop the notions of contiguity and (trivial) extensions of simplicial maps with certain domains. The basic idea is that of \emph{extension-contiguity equivalence} of simplicial maps with domain a (``rectangular") product of intervals.  \secref{sec: extn contiguity} contains some results about operating with with these notions.  The construction of our face group is given in \secref{sec: defn of Face Group}. In \thmref{thm: group} we show that the construction does indeed associate a group to a simplicial complex. In \thmref{thm: Face group abelian} we show the face group is an abelian group (this fact about the face group will follow from our main result \thmref{thm: Face group iso homotopy group}, but we give a direct combinatorial proof here). We also give a combinatorial proof that our face group behaves well with respect to products (\thmref{thm: face group product}).
In \secref{sec: combinatorial face group}, we cast our face group in an intrinsic (to the simplicial complex)  way that strongly resembles the typical description of the edge group.  This lends a somewhat more algorithmic aspect to our construction, comparable to the way in which the edge group is conceived as an algorithmic version of the fundamental group. 
Our main result requires considerable development of results on simplicial approximation tailored to our particular context.  This is done in the next two sections. In \secref{sec: simp approx} we adapt the typical development of simplicial approximation into our specific context of certain triangulations of a rectangle. In 
\secref{sec: Imn vs. ImxIn}, we further adapt this development to suit (trivial) extensions of simplicial maps with domain a simplicial product of intervals.  In \secref{sec: face group iso} we apply all this development to prove our main result \thmref{thm: Face group iso homotopy group}. We finish the paper with a brief \secref{sec: future work} on future work that indicates developments from the basic results shown here. These include applications to our work on digital topology and, separately, a development in the combinatorial setting of a simplicial loop space.    

\subsection*{Acknowledgement of and Comparison with Prior Work.} After this paper was completed and as its companion \cite{LuSc25} was being completed, we became aware of the work of \cite{Gr02}.  There, Grandis associates to a simplicial complex $X$ various other simplicial complexes that model 
spaces of maps into $X$, such as path and loop spaces. In particular, he defines the homotopy group $\widetilde{\pi_2}(X)$ of a simplicial complex $X$ as $\widetilde{\pi_1}(\widetilde{\Omega} X)$ (actually as $\widetilde{\pi_0}(\widetilde{\Omega}^2 X)$), where $\widetilde{\Omega} X$ is a simplicial complex that plays the role of the based loop space of $X$ and $\widetilde{\pi_1}(X)$ is an ``extrinsic" fundamental group that is isomorphic to the ```intrinsic" edge  group of $X$. We use the decorated notations $\widetilde{\pi_i}$ and $\widetilde{\Omega}$ here to distinguish those of \cite{Gr02} from ours in the following discussion. Whilst there are considerable similarities between our work here and some of that in \cite{Gr02} (which has a much wider-ranging scope than our work here), we point here to some of their differences. We also discuss some differences between our development and that of \cite{Gr02} in \cite{LuSc25}.  

First, our work here developed out of our prior work in digital topology.  Since we intend applying our results here to continue our work in digital topology (see \secref{sec: future work}), it is useful for us to have a development that matches the way in which we work in the digital setting: rectangular (finite) domains, left homotopy (using a cylinder object), trivial extensions. The trivial extensions we use here are essentially the same as the ``delays" of \cite{Gr02}, but we adopted their use from Boxer's work in digital topology \cite{Bo99}. 
We chose our face spheres (see \secref{sec: simplicial basics} below for vocabulary and notation) to be finite arrays of vertices in $X$ partly to match the generators of the digital second homotopy groups of \cite{L-M-S-S-T} but also so as to extend directly the edge loops in $X$ that form the edge group, namely finite sequences of vertices. 
 By contrast, a path in \cite{Gr02} has domain $\Z$, and maps such as our face spheres here (``double paths" in \cite{Gr02}, conceived of as loops in $\widetilde{\Omega}X$) have domain $\Z^2$.

Second, Grandis shows an isomorphism  $\widetilde{\pi_2}(X) \cong \pi_2(|X|)$, so our  face group of $X$   is apparently isomorphic to Grandis' 
$\widetilde{\pi_2}(X)$ (\emph{per} \thmref{thm: Face group iso homotopy group}).  But the isomorphism is not so evident without passing through the spatial realization. 
In \cite{LuSc25}, we give a simplicial complex $\Omega X$ similar to the one Grandis gives, and show an isomorphism $F(X) \cong E(\Omega X)$, between the Face Group of $X$ and the Edge Group of $\Omega X$, whilst staying in the combinatorial setting.  Thus, the isomorphism $F(X) \cong \pi_2(|X|) \cong \widetilde{\pi_2}(X)$ may be interpreted as (equivalent to) an isomorphism $E(\Omega X) \cong \widetilde{\pi_1}(\widetilde{\Omega} X)$.  However, as we discuss in \cite{LuSc25}, there are a number of basic differences between the $\widetilde{\Omega} X$ of  \cite{Gr02} and the $\Omega X$ of \cite{LuSc25}. In particular, they are combinatorially different: different vertices and simplices. Overall, a direct, combinatorial correspondence between the $\widetilde{\pi_2}(X)$ of \cite{Gr02} and our $F(X)$ here is not immediate. 

Third, and perhaps most significantly, the finite rectangular domains $I_m \times I_n$ that we use here seem well-suited for approaching $F(X)$ from an algorithmic point of view. For example, given $X$ is a finite simplicial complex, there are only finitely many face spheres contiguous to a given face sphere $f\colon I_m \times I_n \to X$. By contrast, in the set up of \cite{Gr02}, a typical ``double path" representative of a class of  $\widetilde{\pi_2}(X)$ has infinite domain $\Z \times \Z$ and typically there are infinitely many ``double paths" contiguous (or ``linked" in the terminology of \cite{Gr02}) to a given one (even when $X$ is finite).

\section{Simplicial Complexes}\label{sec: simplicial basics}

For a general overview of  material on  simplicial complexes, we refer to \cite{Koz08}.  By a simplicial complex here, we mean an abstract simpicial complex.  
Let $(X, x_0)$ be a based simplicial complex ($x_0$ is some vertex of $X$). Let $I_m$ denote the \emph{(simplicial) interval} of length $m$: 
$I_m = [0, m]_{\mathbb Z}$, the simplical complex with vertices the integers from $0$ to $m$ inclusive and edges given by pairs of consecutive integers.

We will describe a group whose elements are equivalence classes of simplicial maps of the form  $\left(I_{m}\times I_{n}, \partial ( I_{m}\times I_{n}) \right) \to (X, x_0)$. A simplicial map of this form will be referred to as \emph{a face sphere} for reasons that will emerge in the development. Here,  
$I_{m}\times I_{n}$ denotes the \emph{categorical product} of (simplicial) intervals; we describe this structure now.    As an abstract simplicial complex, the vertices of $I_{m}\times I_{n}$ consist of the $(m+1) \times (n+1)$ ordered pairs of integers $(i, j)$ with $i \in I_m$ and $j \in I_n$.  The simplices of  $I_{m}\times I_{n}$ consist of all subsets 
$$\{ (i_{r_1}, j_{s_1}), \ldots, (i_{r_p}, j_{s_p})\}$$
for which $\{ i_{r_1}, \ldots,  i_{r_p}\}$ and $\{ j_{s_1}, \ldots,  j_{s_p}\}$ are both simplices of $I_m$ and $I_n$, respectively.  The maximum number of distinct integers in either $\{ i_{r_1}, \ldots,  i_{r_p}\}$ or $\{ j_{s_1}, \ldots,  j_{s_p}\}$ is $2$.  Thus, the maximum dimension of a simplex of $I_{m}\times I_{n}$ is $3$ (four vertices). For instance, $\{ (0, 0), (1, 0), (0, 1), (1, 1) \}$ is a $3$-simplex of $I_{m}\times I_{n}$ (allowing $m, n \geq 1$). The simplicial complex $X$, on the other hand, is just an ordinary, abstract simplicial complex and is not assumed to have any particular type of simplicial structure or special properties.  

We will show a result about a general (categorical) product $X \times Y$ of simplicial complexes $X$ and $Y$ in \thmref{thm: face group product}. The description of $I_m \times I_n$ is just a special case of the more general definition, which goes as follows. Suppose the vertices of $X$ and $Y$ are $\{ v_i \}_{i \in I}$ and $\{ w_j \}_{j \in J}$, respectively. Then  the vertices of $X \times Y$ consist of all ordered pairs 
$\{ (v_i, w_j)\}_{i\in I, j\in J}$. The simplices of  $X \times Y$ consist of all subsets 
$$\{ (v_{r_1}, w_{s_1}), \ldots, (v_{r_p}, w_{s_p})\}$$
for which $\{ v_{r_1}, \ldots,  v_{r_p}\}$ and $\{ w_{s_1}, \ldots,  w_{s_p}\}$ are simplices of $X$ and of $Y$, respectively.
We will make some use of the categorical nature of this product.  Namely, suppose we have simplicial maps $f \colon K \to X$ and $g \colon K \to Y$.  Then we may use these maps as coordinate maps of a simplicial map
$$(f, g) \colon K \to X\times Y,$$
defined in the usual way as $(f, g)(v) = (f(v), g(v))$ for each vertex $v$ of $K$.   The definition of the simplices of $X \times Y$ is essentially formulated so as to make this map a simplicial map. Furthermore, we have projections onto each coordinate $p_1\colon X \times Y \to X$ and $p_2\colon X \times Y \to Y$ that satisfy $p_1\circ (f, g) = f$ and  $p_2\circ (f, g) = g$.  We give some elementary properties of maps involving general products in \propref{prop: contiguity results1}.

By the \emph{boundary} of $I_m \times I_n$, which we denote by $\partial ( I_{m}\times I_{n})$, we mean the vertices and edges of $\{0, m\} \times I_n \cup I_m \times \{0, n\}$.   The boundary $\partial ( I_{m}\times I_{n})$ is to be treated as a simplicial complex that is a sub-complex of the simplicial complex $I_{m}\times I_{n}$.

\begin{remark}\label{rem: I_m x I_n as clique complex}
Another way to view $I_m \times I_n$ is as the clique complex of a certain graph with vertices those of $I_m \times I_n$. By this, we mean the graph $G_{m, n}$ with vertex set 
$$V\left( G_{m, n}\right) = \left\{ (i, j)\mid i \in I_m \text{ and } j \in I_n \right\}$$
and edge set
$$E\left( G_{m, n} \right) = \left\{ \{ (i, j), (i', j') \} \mid \mathrm{max}\{ |i-i'|, |j-j'| \} \leq 1 \right\}.$$
Then we may view $I_m \times I_n$ as the clique complex of the graph $G_{m, n}$: $n$-cliques are $(n-1)$-simplices.
An example of $I_5 \times I_4$ represented as the clique complex of a graph is given in \figref{fig: Im x In as clique}. [\textbf{Note:} We omit loops at vertices from these figures, although strictly speaking, as we have defined $E\left( G_{m, n} \right)$, there is an edge joining each vertex to itself.]  Whilst the graph $G_{m, n}$ is not a planar graph, this way of representing $I_m \times I_n$ is planar, i.e., flat.
The boundary $\partial ( I_{m}\times I_{n})$ is represented in this way as the ordinary boundary of the square.  It is especially convenient to represent maps $I_m \times I_n \to X$ visually in this way and we will develop this point of view throughout the paper.  
\end{remark}

\begin{figure}
\[
\begin{tikzpicture}[scale=.6]
\foreach \x in {0,...,5} {
 \foreach \y in {0,...,4} {
   \node[inner sep=1.5pt, circle, fill=black]  at (\x,\y) {};
  }
 }
\draw (0,0) grid (5,4);
\draw (0,0) -- (4,4);
\draw (2,0) -- (5,3);
\foreach \i in {0,..., 4} {
  \foreach \j in {0,..., 3} {
    \draw (\i,\j) -- (\i+1,\j+1);
  }
}
\foreach \i in {1,..., 5} {
  \foreach \j in {0,..., 3} {
    \draw (\i,\j) -- (\i-1,\j+1);
  }
}
\end{tikzpicture}
\]
\caption{$I_5 \times I_4$ represented as the clique complex of the graph $G_{5, 4}$.} \label{fig: Im x In as clique}
\end{figure}
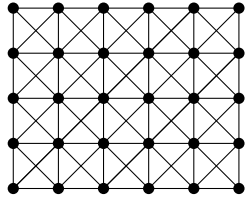

A \emph{face sphere in $X$}, namely a simplicial map $f:  \left(I_m \times I_n, \partial (I_m \times I_n) \right) \to (X, x_0)$, is determined by its vertex map.  We may view the map as a labeling of the vertices of $I_m \times I_n$ with vertices of $X$ in a way that respects simplices.  Continuing the development of \remref{rem: I_m x I_n as clique complex}, we may represent a simplicial map $f\colon  I_m \times I_n \to X$ as a labeling of the vertices of $G_{m, n}$ with vertices of $X$ in such a way that labels of $2$-cliques, $3$-cliques and $4$-cliques in $G_{m, n}$ are simplices of $X$.    

\begin{example}\label{ex: sphere S2}
Let $X$  be the model of the $2$-sphere given by the octahedron which, as a simplicial complex, consists of $6$ vertices, $12$ one-simplices and $8$ two-simplices as illustrated in Figure~\ref{fig: triangulation of S^2}.  Here, we consider this simplicial complex as realized in $\R^3$ with its vertices corresponding to $\pm$ the three standard basis vectors  $\be_1 = (1, 0, 0)$, $\be_2 = (0, 1, 0)$ and $\be_3 = (0, 0, 1)$.  If we choose $x_0 = -\be_1 \in X$ to be the basepoint, then a simplicial map $f:  \left(I_m \times I_n, \partial (I_m \times I_n) \right)  \to (X, x_0)$ corresponds to a labeling of the vertices of $G_{m,n}$ with labels from $\{ \pm \be_1, \pm \be_2, \pm \be_3 \}$, with all vertices around the boundary $\partial (I_m \times I_n)$ labeled with $-\be_1$. In \figref{fig: simplicial map} we have used this ``clique complex of the graph $G_{m, n}$" approach to illustrate a simplicial map $f:  \left(I_5 \times I_4, \partial (I_5 \times I_4) \right)  \to (X, x_0)$.
\end{example}  

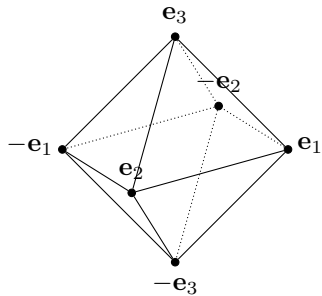
\begin{figure}
\tikzset{vertex/.style={circle,draw,fill,inner sep=0pt,minimum size=1mm}}
\[
\vcenter{\hbox{
\begin{tikzpicture}[scale=1.5]
\node[vertex,label={[right]$\be_1$}] (p1) at (1,0,0) {};
\node[vertex,label={[left]$-\be_1$}] (m1) at (-1,0,0) {};
\node[vertex,label=$\be_3$] (p3) at (0,1,0) {};
\node[vertex,label={[below=.1cm]$-\be_3$}] (m3) at (0,-1,0) {};
\node[vertex,label=$\be_2$] (p2) at (0,0,1) {};
\node[vertex,label=$-\be_2$] (m2) at (0,0,-1) {};
\foreach \x in {p2,p3,m3} {
 \draw (p1) -- (\x) -- (m1);
}
\draw[densely dotted] (p1) -- (m2) -- (m1);
\draw (p3) -- (p2) -- (m3);
\draw[densely dotted] (p3) -- (m2) -- (m3);
\end{tikzpicture}
}}
\] 
\caption{The $2$-sphere as a simplicial complex $X$.\label{fig: triangulation of S^2}}
\end{figure}

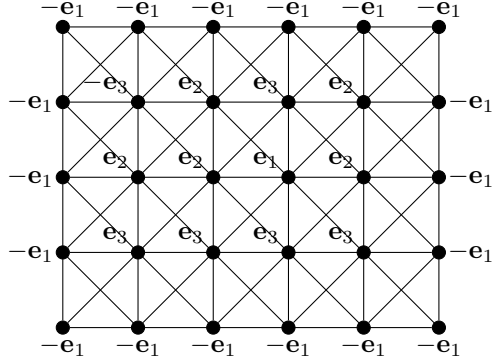
\begin{figure}
\centering
$$
\begin{tikzpicture}[scale=1]

\node[inner sep=1.75pt, circle, fill=black] (v00) at (0,0) [draw] {};
\node[inner sep=1.75pt, circle, fill=black] (v01) at (0,1) [draw] {};
\node[inner sep=1.75pt, circle, fill=black] (v02) at (0,2) [draw] {};
\node[inner sep=1.75pt, circle, fill=black] (v03) at (0,3) [draw] {};
\node[inner sep=1.75pt, circle, fill=black] (v04) at (0,4) [draw] {};
\node[inner sep=1.75pt, circle, fill=black] (v10) at (1,0) [draw] {};
\node[inner sep=1.75pt, circle, fill=black] (v11) at (1,1) [draw] {};
\node[inner sep=1.75pt, circle, fill=black] (v12) at (1,2) [draw] {};
\node[inner sep=1.75pt, circle, fill=black] (v13) at (1,3) [draw] {};
\node[inner sep=1.75pt, circle, fill=black] (v14) at (1,4) [draw] {};
\node[inner sep=1.75pt, circle, fill=black] (v20) at (2,0) [draw] {};
\node[inner sep=1.75pt, circle, fill=black] (v21) at (2,1) [draw] {};
\node[inner sep=1.75pt, circle, fill=black] (v22) at (2,2) [draw] {};
\node[inner sep=1.75pt, circle, fill=black] (v23) at (2,3) [draw] {};
\node[inner sep=1.75pt, circle, fill=black] (v24) at (2,4) [draw] {};
\node[inner sep=1.75pt, circle, fill=black] (v30) at (3,0) [draw] {};
\node[inner sep=1.75pt, circle, fill=black] (v31) at (3,1) [draw] {};
\node[inner sep=1.75pt, circle, fill=black] (v32) at (3,2) [draw] {};
\node[inner sep=1.75pt, circle, fill=black] (v33) at (3,3) [draw] {};
\node[inner sep=1.75pt, circle, fill=black] (v34) at (3,4) [draw] {};
\node[inner sep=1.75pt, circle, fill=black] (v40) at (4,0) [draw] {};
\node[inner sep=1.75pt, circle, fill=black] (v41) at (4,1) [draw] {};
\node[inner sep=1.75pt, circle, fill=black] (v42) at (4,2) [draw] {};
\node[inner sep=1.75pt, circle, fill=black] (v43) at (4,3) [draw] {};
\node[inner sep=1.75pt, circle, fill=black] (v44) at (4,4) [draw] {};
\node[inner sep=1.75pt, circle, fill=black] (v50) at (5,0) [draw] {};
\node[inner sep=1.75pt, circle, fill=black] (v51) at (5,1) [draw] {};
\node[inner sep=1.75pt, circle, fill=black] (v52) at (5,2) [draw] {};
\node[inner sep=1.75pt, circle, fill=black] (v53) at (5,3) [draw] {};
\node[inner sep=1.75pt, circle, fill=black] (v54) at (5,4) [draw] {};

\draw (v00) -- (v50);
\draw (v01) -- (v51);
\draw (v02) -- (v52);
\draw (v03) -- (v53);
\draw (v04) -- (v54);
\draw (v00) -- (v04);
\draw (v10) -- (v14);
\draw (v20) -- (v24);
\draw (v30) -- (v34);
\draw (v40) -- (v44);
\draw (v50) -- (v54);
\draw (v03) -- (v14);
\draw (v02) -- (v24);
\draw (v01) -- (v34);
\draw (v00) -- (v44);
\draw (v10) -- (v54);
\draw (v20) -- (v53);
\draw (v30) -- (v52);
\draw (v40) -- (v51);

\foreach \i in {0, 1, 2, 3, 4, 5} {
\node[below]  at (\i,0) {$-\mathbf{e}_1$};
\node[above]  at (\i,4) {$-\mathbf{e}_1$};
}
\foreach \j in {1, 2, 3} {
\node[left]  at (0,\j) {$-\mathbf{e}_1$};
\node[right]  at (5,\j) {$-\mathbf{e}_1$};
}
\node[above left]  at (1,1) {$\mathbf{e}_3$};
\node[above left]  at (2,1) {$\mathbf{e}_3$};
\node[above left]  at (3,1) {$\mathbf{e}_3$};
\node[above left]  at (4,1) {$\mathbf{e}_3$};
\node[above left]  at (1,2) {$\mathbf{e}_2$};
\node[above left]  at (2,2) {$\mathbf{e}_2$};
\node[above left]  at (3,2) {$\mathbf{e}_1$};
\node[above left]  at (4,2) {$\mathbf{e}_2$};
\node[above left]  at (1,3) {$-\mathbf{e}_3$};
\node[above left]  at (2,3) {$\mathbf{e}_2$};
\node[above left]  at (3,3) {$\mathbf{e}_3$};
\node[above left]  at (4,3) {$\mathbf{e}_2$};

\foreach \i in {1,..., 5} {
  \foreach \j in {0,..., 3} {
    \draw (\i,\j) -- (\i-1,\j+1);
  }
}

\end{tikzpicture}
$$
\caption{A face sphere in $X$, in this case a simplicial map  $f: \left(I_5 \times I_4, \partial (I_5 \times I_4) \right) \to (S^2,-\mathbf{e}_1)$ }\label{fig: simplicial map}
\end{figure}

In the sequel, we will refer to this style of representing a face sphere in $X$---a labeling of the vertices of $G_{m, n}$ with vertices from $X$---as ``array style." The vertex map that defines the face sphere is essentially an $(m+1) \times (n+1)$ matrix with entries from the vertices of $X$ that satisfy certain coherence relations.   

\section{Trivial Extension and Extension-Contiguity Equivalence of Maps}\label{sec: extn contiguity}

Simplicial maps $f, g\colon K \to X$ of (abstract) simplicial complexes $K$ and $X$ are \emph{contiguous}, denoted by $f \sim g$, if  $f(\sigma) \cup g(\sigma)$ is a simplex of $X$ whenever $\sigma$ is a simplex of $K$.  Equivalently, we might require that $f(\sigma) \cup g(\sigma)$ be contained in a simplex of $X$.   We are interested in this notion when $K$ is  $I_m \times I_n$ and the maps $f$ and $g$ both map the boundary $\partial(I_m \times I_n)$ to the basepoint of $X$.  Contiguity is evidently a reflexive and symmetric relation.  It is not transitive in general, but there is an associated equivalence relation on maps $K \to X$.  Namely, we say simplicial maps $f$ and $g$ are \emph{contiguity equivalent}, and write $f \simeq g$, if there is a finite sequence of contiguities $f \sim f_1 \sim \cdots \sim f_n \sim g$.  Then contiguity equivalence is an equivalence relation on the set of all simplicial maps $K \to X$ for fixed $K$ (and $X$).  The simplicial maps we consider will actually be maps of pairs $f, g\colon \left(I_m \times I_n, \partial (I_m \times I_n) \right) \to (X, x_0)$ and we will require that each map in the sequence of contiguity-equivalent $f$ and $g$ also be of this form.  Strictly speaking, then, our contiguity equivalences  will be relative the boundary of $I_{m}\times I_{n}$.  Finally, we note that when checking a pair of maps to be contiguous or not, it is sufficient to check that $f(\sigma) \cup g(\sigma)$ is a simplex of $L$ whenever $\sigma$ is a simplex of $K$ \emph{that is not a face of some larger simplex}.  This follows since, if $\sigma'$ is a face of a simplex $\sigma$ such that $f(\sigma) \cup g(\sigma)$ is a simplex of $X$, then  $f(\sigma') \cup g(\sigma')$ is a subset of $f(\sigma) \cup g(\sigma)$ and thus a simplex of $X$. In our case, this means that we may confirm two maps $f, g \colon \left(I_m \times I_n, \partial (I_m \times I_n) \right) \to (X, x_0)$ to be contiguous by checking that $f(\sigma) \cup g(\sigma)$ is a simplex of $X$ for each $3$-simplex $\sigma$ of $I_m \times I_n$.  

We now record some basic results on contiguity equivalence. These are included mainly for convenience of reference.  
The first result provides us with a way of working ``locally" with contiguity equivalences.      

\begin{lemma}\label{localized contiguity equiv}
Let  $R$ be a \emph{rectangular sub-complex} of $I_m \times I_n$. Namely, suppose that $R = [p, q] \times [r, s] \subseteq I_{m}\times I_{n}$ for integers $p, q, r, s$ with  $0 \leq p \leq q \leq m$ and $0 \leq r \leq s \leq n$. Let $f \colon \left(I_m \times I_n, \partial (I_m \times I_n) \right) \to (X, x_0)$ be a simplicial map for which $f_R$, the restriction of $f$ to the sub-complex $R$, is  a map of pairs $f_R: (R, \partial R) \to (X, x_0)$. Let $g \colon (R, \partial R) \to (X, x_0)$ be a map on the sub-complex such that $f_R\simeq g$ by a contiguity equivalence relative the boundary $\partial R$.
Then the map $A \colon \left(I_m \times I_n, \partial (I_m \times I_n) \right) \to (X, x_0)$ defined on vertices by:
\[ A(i,j) = \begin{cases} g(i,j) &\text{ if } (i,j) \in R \\
f(i,j) &\text{ if } (i,j)\not \in R \\
\end{cases} \]
extends to a simplicial map of $I_m \times I_n$ and we have a contiguity equivalence $f \simeq A$.
\end{lemma}

\begin{proof}
The map $A$ is assured to give a simplical map because the maps $f$ and $g$ agree on the subcomplex $\partial R$, the boundary of the rectangle $R$.  This boundary is the intersection of two subcomplexes of $I_{m} \times I_{n}$, namely $R$ and the union of the complement of $R$ with the boundary $\partial R$. These two subcomplexes have union equal to $I_{m} \times I_{n}$ and do not intersect other than in $\partial R$.  Other than the simplices of $\partial  R$, no simplex of $I_{m} \times I_{n}$ belongs to both $R$ and the union of the complement of $R$ and its boundary.  It follows that $A$ is a simplicial map since both $f$ and $g$ are. 

Furthermore, suppose we have a sequence of contiguities $f_R = f_1 \sim \cdots \sim f_k = g$ of maps of $R$ relative the boundary of $R$. Define $A_t  \colon \left(I_m \times I_n, \partial (I_m \times I_n) \right) \to (X, x_0)$ as we did for $A$ only using $f_t$ on $R$ in place of $f_R$, for each $t=1, \ldots, k$.  Then each contiguity $f_t \sim f_{t+1}$ extends to a contiguity $A_t \sim A_{t+1}$ with $A_t(i, j) = A_{t+1}(i, j) = f(i, j)$ for each $(i,j)\not \in R$.  Thus, we have contiguities $f = A_1 \sim \cdots \sim A_k = A$ as desired.
\end{proof}

When we apply \lemref{localized contiguity equiv}, we will refer to ``patching in" the (local) contiguity equivalence on the sub-region $R$ to obtain one on the larger domain $I_m \times I_n$ by leaving values of the maps on vertices outside $R$ unchanged. 
Also,  when making use of \lemref{localized contiguity equiv}, we will frequently employ the device of translating the sub-complex $R$ to a standard rectangle $I_{q-p} \times I_{s-r}$.  We fix notation for such a translation now.

\begin{definition}\label{def: translation Tpq}
Let $T_{p, q} \colon \Z \times \Z \to \Z \times \Z$ be the translation defined by $T_{p, q}(i, j) = (i-p, j-q)$ for each  $(i, j) \in \Z \times \Z$.  This is evidently a simplicial map.
\end{definition}

\begin{proposition}\label{prop: contiguity results1} 
    \begin{itemize}
\item[(a)] Given pairs of contiguity equivalent simplicial maps $f \simeq f' \colon X \to Y$ and $g \simeq g' \colon Y \to Z$, their compositions are contiguity equivalent: we have $g\circ f \simeq g' \circ f' \colon X \to Z$.
\item[(b)] Given pairs of contiguity equivalent simplicial maps $f \simeq f' \colon K_1 \to X_1$ and $g \simeq g' \colon K_2 \to X_2$, their products are contiguity equivalent: we have $f\times g \simeq f' \times g' \colon K_1 \times K_2 \to X_1 \times X_2$.
\item[(c)] Given pairs of contiguity equivalent simplicial maps $f \simeq f' \colon K \to X$ and $g \simeq g' \colon K \to Y$, we have a contiguity equivalence $(f, g) \simeq (f', g') \colon K \to X \times Y$.  
    \end{itemize}
\end{proposition}

\begin{proof}
(a) First suppose that we have $f\sim f'$ and $g \sim g'$. Say $\sigma$ is a simplex of $X$.  Because $f \sim f'$, we have $f(\sigma) \cup f'(\sigma)$ contained in some simplex $\sigma'$ of $Y$. Then we have $g\circ f(\sigma) \cup g'\circ f'(\sigma) \subseteq g(\sigma') \cup g'(\sigma')$, which is contained in some simplex $\sigma''$ of $Z$, as $g \sim g'$. Thus, we have $g\circ f \sim g' \circ f'$.  Now, given contiguity equivalences $f \sim f_1 \sim \cdots \sim f'$ and $g \sim g_1 \sim \cdots \sim g'$, we may write a contiguity equivalence of the form
$$g\circ f \sim g\circ f_1 \sim \cdots \sim g\circ f' \sim g_1\circ f' \sim \cdots \sim g'\circ f'.$$

(b)  First suppose that we have $f\sim f'$ and $g \sim g'$. Say $\sigma$ is a simplex of $K_1 \times K_2$. This means
$\sigma = \left\{ (x_0, y_0), \ldots, (x_n, y_n) \right\}$ with $\sigma_1 := \{ x_0, \ldots, x_n \}$ a simplex of $K_1$ and $\sigma_2 := \{ y_0, \ldots, y_n \}$ a simplex of $K_2$.
Then we have 
$$\begin{aligned}    
(f\times g)(\sigma) &\cup (f'\times g')(\sigma) \\
&= 
\left\{ (f(x_0), g(y_0)), \ldots, (f(x_n), g(y_n)),
(f'(x_0), g'(y_0)), \ldots, (f'(x_n), g'(y_n))\right\}.
\end{aligned}$$
Because $f\sim f'$, we have 
$$f(\sigma_1) \cup f'(\sigma_1) = \{ f(x_0), \ldots, f(x_n), 
f'(x_0), \ldots, f'(x_n)\}$$
is a simplex of $X_1$. Because $g\sim g'$, we have 
$$g(\sigma_2) \cup g'(\sigma_2) = \{ g(y_0), \ldots, g(y_n), 
g'(y_0), \ldots, g'(y_n)\}$$
is a simplex of $X_2$. It follows that $(f\times g)(\sigma) \cup (f'\times g')(\sigma)$ is a simplex of $X_1 \times X_2$: we have $f\times g \sim f'\times g'$. 

Using this, suppose given contiguity equivalences $f \sim f_1 \sim \cdots \sim f'$ and $g \sim g_1 \sim \cdots \sim g'$.  Then  we may write a contiguity equivalence of the form
$$f\times g \sim f_1\times g \sim \cdots \sim f'\times g \sim f'\times g_1 \sim \cdots \sim f'\times g'.$$

(c) First suppose that we have $f\sim f'\colon K \to X$ and $g \sim g'\colon K \to Y$. Say $\sigma = \{ v_1, \ldots, v_r \}$ a simplex of $K$.
Then $f(\sigma) \cup f'(\sigma)$ is a simplex of $X$ and $g(\sigma) \cup g'(\sigma)$ is a simplex of $Y$.  We have 
$$(f, g)(\sigma) \cup (f', g')(\sigma) = \left\{  \left( f(v_i), g(v_i) \right)_{i=1, \ldots, r} , \left( f'(v_i), g'(v_i) \right)_{i=1, \ldots, r}  \right\},$$
which is a simplex of $X \times Y$ since the union of all the first coordinates is  $f(\sigma) \cup f'(\sigma)$, a simplex of $X$, whilst the union of all the second coordinates is $g(\sigma)\cup g'(\sigma)$, a simplex of $Y$. Thus we have $(f, g) \sim (f', g')\colon K \to X \times Y$.

Now suppose given contiguity equivalences $f \sim f_1 \sim \cdots \sim f'$ and $g \sim g_1 \sim \cdots \sim g'$.  Then  we may write a contiguity equivalence of the form
$$(f, g) \sim (f_1, g) \sim \cdots \sim (f', g) \sim (f',  g_1) \sim \cdots \sim (f', g')$$
and the assertion follows.     
\end{proof}

For our purposes, we need to be able to compare not only simplicial maps from a fixed $I_m \times I_n$ to a given simplicial complex $X$, but also maps from different sized $I_m \times I_n$ to $X$.  For this, we use the notion of trivial extension and the more general notions of (repeated) row- and column-doubling.      We develop these ideas and some basic ways of operating with face spheres in the next several results.   

\begin{definition}\label{def: extensions}
For each $m$, define simplicial maps $\alpha_i \colon I_{m+1} \to I_m$ for $i = 0, \ldots, m$ as follows:
$$\alpha_i(s) = 
\begin{cases}
s & 0 \leq s \leq i \\
 & \\
s-1 & i+1 \leq s \leq m+1 \\
\end{cases}
$$
Now for a simplicial map $l\colon I_m \to X$, i.e., a \emph{path in $X$}, we refer to a composition $l \circ \alpha_i \colon I_{m+1} \to X$ as an \emph{extension} of $l$. Whereas the map $\alpha_i$ shrinks the interval, the composition $l \circ \alpha_i$ is the path obtained from $l$ by repeating the $i$th vertex.  More generally, if $I = \{ i_1, \ldots, i_r \}$ is a sequence with $0 \leq i_t \leq m+t-1$ for each $1 \leq t \leq r$, we write 
$$\alpha_I := \alpha_{i_1} \circ \alpha_{i_2} \circ \cdots \circ \alpha_{i_r} \colon I_{m+r} \to I_m$$
and also refer to $l\circ \alpha_I\colon I_{m+r} \to X$ as an extension of $l$. It is the path obtained from $l$ by repeating the $i_r$th vertex, then repeating the $i_{r-1}$st vertex of that extended path, and so-on. If $I = \{ i, \ldots, i \}$ ($r$-times), then we write $\alpha^r_i$ for $\alpha_I = \alpha_i \circ \cdots \circ \alpha_i$. Then $l\circ \alpha^r_i \colon I_{m+r} \to X$ is the path obtained from $l$ by repeating $r$-times the $i$th vertex of $l$. We distinguish the case in which $i=m$ by referring to $l\circ \alpha^r_m$ as a \emph{trivial extension} of the path $l \colon I_m \to X$.
\end{definition}

Now suppose we have  a face sphere $f \colon \left( I_{m} \times I_{n}, \partial (I_{m} \times I_{n}) \right) \to (X, x_0)$. Then the composition
$$f \circ (\alpha^r_i \times \alpha^s_j) \colon \left( I_{m+r} \times I_{n+s}, \partial (I_{m+r} \times I_{n+s}) \right) \to (X, x_0)$$
is the face sphere of size $(m+r)\times (n+s)$ obtained from $f$---when viewed ``array-style" as an $(m+1)\times (n+1)$ array of values in $X$, as in \figref{fig: simplicial map}---by repeating the $i$th column of values $r$ times and the $j$th row of values $s$ times.
In particular, we will refer to the compositions $f \circ (\alpha^r_m \times \alpha^s_n)$ as \emph{trivial extensions of the face sphere $f$}.  We will often denote a trivial extension of $f$ by $\overline{f}$.  Thus, if we have 
$\overline{f} = f \circ (\alpha^r_m \times \alpha^s_n) \colon \left(I_{m+r} \times I_{n+s}, \partial (I_{m+r} \times I_{n+s}) \right) \to (X,x_0)$ a trivial extension of $f \colon \left(I_m \times I_n, \partial (I_m \times I_n) \right) \to (X,x_0)$, then the vertex map of $\overline{f}$ is given by
\[ \overline{f}(i, j) = \begin{cases}
f(i, j) & \text{ if $(i, j)\in I_m \times I_n$,} \\
x_0 & \text{ otherwise.}
\end{cases} \]

We may add the following to our basic results about contiguity equivalence.

\begin{proposition}\label{prop: contiguity results2} 
Using the notation and vocabulary from above, we have the following.
    \begin{itemize}
\item[(a)] We have a contiguity equivalence $\alpha_i \simeq \alpha_j \colon I_{m+1} \to I_m$ for any $0 \leq i < j \leq m$.
\item[(b)] Suppose $\alpha_I, \alpha_J \colon I_{m+r} \to I_m$  are any maps of the form
$$\alpha_I := \alpha_{i_1} \circ \cdots \circ \alpha_{i_r} \qquad \mathrm{and} \qquad \alpha_J := \alpha_{j_1} \circ \cdots \circ \alpha_{j_r},$$
and $\alpha_{I'}, \alpha_{J'} \colon I_{n+s} \to I_n$  are any maps of the form
$$\alpha_{I'} := \alpha_{i'_1} \circ \cdots \circ \alpha_{i'_s} \qquad \mathrm{and} \qquad \alpha_{J'} := \alpha_{j'_1} \circ \cdots \circ \alpha_{j'_s}.$$
Then we have a contiguity equivalence  $\alpha_I \times \alpha_{I'} \simeq \alpha_J \times \alpha_{J'} \colon I_{m+r} \times I_{n+s}  \to I_m \times I_n$. 
    \end{itemize}
\end{proposition}

\begin{proof}
(a) It is sufficient to show that we have $\alpha_i \sim \alpha_{i+1}$ for each $0 \leq i \leq m-1$. To this end, consider the typical simplex $\sigma = 
\{ s, s+1 \}$ of $I_{m+1}$, some  $0 \leq s \leq m$. Then a direct check gives that
$$\begin{aligned}
    \alpha_i(\sigma) \cup \alpha_{i+1}(\sigma) & = \left\{ \alpha_i(s), \alpha_i(s+1), \alpha_{i+1}(s), \alpha_{i+1}(s+1)\right\}\\
    &= \begin{cases}
        \{ s, s+1, s, s+1 \} = \{s, s+1\} & 0 \leq s \leq i-1 \\
        \{ i, i+1, i, i \} = \{i, i+1\} & s = i \\
         \{ i, i+1, i+1, i+1 \} = \{i, i+1\} & s = i+1 \\
         \{ s-1, s, s-1, s \} = \{s-1, s\} & i+2 \leq s \leq m. 
    \end{cases}
\end{aligned}$$
In all cases, we see that $\alpha_i(\sigma) \cup \alpha_{i+1}(\sigma)$ is a simplex of $I_m$. Namely, we have $\alpha_i \sim \alpha_{i+1}$.

(b) This follows directly from  \propref{prop: contiguity results1} and part (a) above.
\end{proof}

In the next result, by $\alpha_j^0$ we intend the identity map $\alpha_j^0 = \mathrm{id} \colon I_m \to I_m$.

\begin{corollary}\label{cor: different doublings}
 Let  $f \colon I_m \times I_n \to X$ be a simplicial map. For given $r, s \geq 0$, we have a contiguity equivalence
 $$f \circ (\alpha^r_i \times \alpha^s_j) \simeq f \circ (\alpha^r_k \times \alpha^s_l)\colon I_{m+r} \times I_{n+s} \to X$$
 for each $0 \leq i, k \leq m$ and each $0 \leq j, l \leq n$. 
\end{corollary}

\begin{proof}
The assertion is that repeating different columns the same number of times and/or repeating different rows the same number of times leads to contiguity equivalent maps. 

The result follows directly from part (a) of \propref{prop: contiguity results1} and part (b) of \propref{prop: contiguity results2}  
\end{proof}

Now we give our basic notion of equivalence for face spheres. 

\begin{definition}
Given simplicial maps $f\colon  \left(I_m \times I_n, \partial (I_m \times I_n) \right) \to (X,x_0)$ and $g\colon \left(I_{m'} \times I_{n'}, \partial (I_{m'} \times I_{n'}) \right) \to (X,x_0)$, we  say that $f$ and $g$ are \emph{extension-contiguity equivalent}, and write $f \approx g$, when for some $\bar m \ge \max(m,m')$ and $\bar n \ge \max(n,n')$ there is a contiguity equivalence (relative the boundary)  
$$f\circ (\alpha_m^{\overline{m} - m} \times \alpha_n^{\overline{n} - n}) \simeq g\circ (\alpha_{m'}^{\overline{m} - m'} \times \alpha_{n'}^{\overline{n} - n'})$$
of trivial extensions of $f$ and $g$.
\end{definition}

\begin{theorem}\label{thm: extn-cont equiv}
Extension-contiguity equivalence of maps is an equivalence relation on the set of maps $\left(I_m \times I_n, \partial (I_m \times I_n) \right) \to (X,x_0)$ (all shapes and sizes of rectangle).
\end{theorem}

\begin{proof}
Reflexivity and symmetry follow immediately because contiguity equivalence (relative the boundary) of maps is an equivalence relation. If $f \simeq g$, then we have a contiguity equivalence 
$f \circ (\alpha^r_m \times \alpha^s_n) \simeq g \circ (\alpha^r_m \times \alpha^s_n)$ of trivial extensions, for any $r$ and $s$, by
part (a) of \propref{prop: contiguity results1}
So, suppose we have maps $f_1\colon \left(I_{m_1} \times I_{n_1}, \partial (I_{m_1} \times I_{n_1}) \right) \to (X,x_0)$, $f_2\colon \left(I_{m_2} \times I_{n_2}, \partial (I_{m_2} \times I_{n_2}) \right) \to (X,x_0)$ and $f_3\colon \left(I_{m_3} \times I_{n_3}, \partial (I_{m_3} \times I_{n_3}) \right) \to (X,x_0)$, and that $f_1\approx f_2 \approx f_3$.  
Because $f_1\approx f_2$, we have contiguity equivalent trivial extensions $f_1 \circ (\alpha^{r_1}_{m_1} \times \alpha^{s_1}_{n_1}) \simeq  f_2 \circ (\alpha^{r_2}_{m_2} \times \alpha^{s_2}_{n_2})$ for some $r_1, r_2, s_1$ and $s_2$. Note that this entails the equalities $m_1+r_1 = m_2 + r_2$ and $n_1+s_1 = n_2 + s_2$.  Likewise, we have contiguity equivalent trivial extensions $f_2 \circ (\alpha^{r'_2}_{m_2} \times \alpha^{s'_2}_{n_2}) \simeq  f_3 \circ (\alpha^{r_3}_{m_3} \times \alpha^{s_3}_{n_3})$ for some $r'_2, r_3, s'_2$ and $s_3$.
Now part (a) of \propref{prop: contiguity results1} gives contiguity equivalences
$$f_1 \circ (\alpha^{r_1}_{m_1} \times \alpha^{s_1}_{n_1}) \circ (\alpha^{r'_2}_{m_2+r_2} \times \alpha^{s'_2}_{n_2+s_2}) \simeq  f_2 \circ (\alpha^{r_2}_{m_2} \times \alpha^{s_2}_{n_2}) \circ (\alpha^{r'_2}_{m_2+r_2} \times \alpha^{s'_2}_{n_2+s_2})$$
and
$$f_2 \circ (\alpha^{r'_2}_{m_2} \times \alpha^{s'_2}_{n_2}) \circ (\alpha^{r_2}_{m_2+r'_2} \times \alpha^{s_2}_{n_2+s'_2}) \simeq  f_3 \circ (\alpha^{r_3}_{m_3} \times \alpha^{s_3}_{n_3}) \circ (\alpha^{r_2}_{m_2+r'_2} \times \alpha^{s_2}_{n_2+s'_2}).$$
Furthermore, we may write
$$ (\alpha^{r_1}_{m_1} \times \alpha^{s_1}_{n_1}) \circ (\alpha^{r'_2}_{m_2+r_2} \times \alpha^{s'_2}_{n_2+s_2}) =  (\alpha^{r_1}_{m_1} \times \alpha^{s_1}_{n_1}) \circ (\alpha^{r'_2}_{m_1+r_1} \times \alpha^{s'_2}_{n_1+s_1}) = \alpha^{r_1+r'_2}_{m_1} \times \alpha^{s_1+s'_2}_{n_1},$$
$$(\alpha^{r_2}_{m_2} \times \alpha^{s_2}_{n_2}) \circ (\alpha^{r'_2}_{m_2+r_2} \times \alpha^{s'_2}_{n_2+s_2}) =   \alpha^{r_2+r'_2}_{m_2} \times \alpha^{s_2+s'_2}_{n_2} = (\alpha^{r'_2}_{m_2} \times \alpha^{s'_2}_{n_2}) \circ (\alpha^{r_2}_{m_2+r'_2} \times \alpha^{s_2}_{n_2+s'_2}),$$
and
$$(\alpha^{r_3}_{m_3} \times \alpha^{s_3}_{n_3}) \circ (\alpha^{r_2}_{m_2+r'_2} \times \alpha^{s_2}_{n_2+s'_2}) =  (\alpha^{r_3}_{m_3} \times \alpha^{s_3}_{n_3}) \circ (\alpha^{r_2}_{m_3+r_3} \times \alpha^{s_2}_{n_3+s_3}) = \alpha^{r_3+r_2}_{m_3} \times \alpha^{s_3+s_2}_{n_3}.$$
Thus we have contiguity equivalences of trivial extensions
$$f_1\circ (\alpha^{r_1+r'_2}_{m_1} \times \alpha^{s_1+s'_2}_{n_1}) \simeq f_2 \circ  (\alpha^{r_2+r'_2}_{m_2} \times \alpha^{s_2+s'_2}_{n_2}) \simeq f_3 \circ (\alpha^{r_3+r_2}_{m_3} \times \alpha^{s_3+s_2}_{n_3}).$$
Since contiguity equivalence is transitive, we have $\bar{f_1} \simeq \bar{f_3}$ and the result follows.
\end{proof}

\section{Definition of $F(X,x_0)$, the Face Group of a Simplicial Complex}\label{sec: defn of Face Group}

We now give the definition of our Face Group and establish its basic general properties.

\begin{definition}
Given a based simplicial complex $(X,x_0)$, the \emph{Face Group} of $(X,x_0)$, written $F(X,x_0)$, is the set of equivalence classes of simplicial maps of the form 
$$\left(I_{m} \times I_{n}, \partial (I_{m} \times I_{n}) \right)\to (X,x_0)$$
(namely, face spheres in $X$) for all  $I_{m} \times I_{n}$, modulo the equivalence relation of extension-contiguity equivalence.  We write the equivalence class of a face sphere $f$ as $[f] \in F(X, x_0)$.
\end{definition}

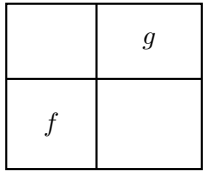
\begin{figure}
\[
\begin{tikzpicture}[scale=.2]
\draw[thick] (0,0) rectangle (13,11);
\draw[thick] (0,6) -- (13,6);
\draw[thick] (6,0) -- (6,11);

\node at (3,3) {$f$};

\node at (9.5,8.5) {$g$};
\end{tikzpicture}
\]
\caption{Schematic of the product $f\cdot g$ of two maps $f$ and $g$. }\label{cdotfig}
\end{figure}

The group operation in $F(X,x_0)$ is induced by the following operation on maps. Let $f\colon \left(I_{m} \times I_{n}, \partial (I_{m} \times I_{n}) \right) \to (X,x_0)$ and $g\colon \left(I_{r} \times I_{s}, \partial (I_{r} \times I_{s}) \right) \to (X,x_0)$ be simplicial maps.  Recall our translation notation $T_{p, q}$ from \defref{def: translation Tpq}. Define $f\cdot g\colon \left(I_{m+r+1} \times I_{n+s+1}, \partial (I_{m+r+1} \times I_{n+s+1}) \right)\to (X, x_0)$ on vertices by
$$
(f\cdot g)(i,j)=
    \begin{cases}
        f(i,j) & \text{if } (i,j)\in I_m\times I_n\\
        g\circ T_{m+1, n+1}(i, j) & \text{if } (i,j)\in [m+1,m+r+1]\times [n+1,n+s+1]\\
        x_0 & \text{otherwise}
    \end{cases}
$$
and extend as a simplicial map over each simplex of $I_{m+r+1} \times I_{n+s+1}$. 
See Figure~\ref{cdotfig} for an illustration.
In this figure and subsequent similar ones, the vertices of $I_{m} \times I_{n}$---and only those vertices---are intended to be contained in the lower-left rectangle, which thus would actually be a rectangle $[-0.5, m+0.5]\times [-0.5, n+0.5]$.  Then the vales of $f$ are assigned to the vertices in this square. As $f$ is a face sphere, note that the vertices immediately inside this square will be assigned $x_0$.  All vertices in any rectangle left blank in such a picture are assigned the value $x_0$. We abuse notation slightly by labeling the upper-right square with $g$.  Strictly speaking, these vertices are assigned the values of the map $g\circ T_{m+1, n+1}$.    

\begin{proposition}\label{prop: operation well-defined} Suppose we have simplicial maps 
$$f_1\colon \left( I_{m_1}\times I_{n_1}, \partial (I_{m_1}\times I_{n_1}) \right) \to (X, x_0), \quad   f_2\colon \left( I_{m_2}\times I_{n_2}, \partial (I_{m_2}\times I_{n_2}) \right)\to (X, x_0),$$
and
$$g_1\colon \left( I_{r_1}\times I_{s_1}, \partial (I_{r_1}\times I_{s_1}) \right)\to (X, x_0), \quad g_2\colon \left( I_{r_2}\times I_{s_2}, \partial (I_{r_2}\times I_{s_2}) \right) \to (X, x_0)$$
with $f_1\approx f_2$ and $g_1\approx g_2$. Then $f_1\cdot g_1\approx f_2 \cdot g_2$.
\end{proposition}

\begin{proof} Let $f_1,f_2,g_1$, and $g_2$ satisfy the above. Then there are trivial extensions $\bar{f_1}, \bar{f_2}\colon I_{\bar{m}} \times I_{\bar{n}} \to X$ of $f_1$ and $f_2$, respectively, that are contiguity equivalent relative the boundary.  Similarly, there are trivial extensions $\bar{g_1}, \bar{g_2}\colon I_{\bar{r}} \times I_{\bar{s}} \to X$ of $g_1$ and $g_2$, respectively, that are contiguity equivalent relative the boundary. 

Notice that $\overline{f_1} \cdot \overline{g_1}$ may be written as a composition
$$\overline{f_1} \cdot \overline{g_1} = 
(f_1 \cdot g_1) \circ (\alpha_{m_1}^{\overline{m} - m_1} \times \alpha_{n_1}^{\overline{n} - n_1})
\circ (\alpha_{\overline{m}+r_1+1}^{\overline{r} - r_1} \times \alpha_{\overline{n}+s_1+1}^{\overline{s} - s_1}).$$
For the first map applied in this composition,
part (b) of \propref{prop: contiguity results2} gives a contiguity equivalence
$$(\alpha_{\overline{m}+r_1+1}^{\overline{r} - r_1} \times \alpha_{\overline{n}+s_1+1}^{\overline{s} - s_1})
\simeq (\alpha_{m_1+r_1+1}^{\overline{r} - r_1} \times \alpha_{n_1+s_1+1}^{\overline{s} - s_1})
\colon I_{\overline{m} + \overline{r}+1} \times I_{\overline{n} + \overline{s}+1} \to I_{\overline{m} + r_1+1} \times I_{\overline{n} + s_1+1}$$
For the middle map, we have
$$(\alpha_{m_1}^{\overline{m} - m_1} \times \alpha_{n_1}^{\overline{n} - n_1})
\simeq (\alpha_{m_1+r_1+1}^{\overline{m} - m_1} \times \alpha_{n_1+s_1+1}^{\overline{n} - n_1})
\colon I_{\overline{m} + r_1+1} \times I_{\overline{n} + s_1+1} \to I_{m_1 + r_1+1} \times I_{n_1 + s_1+1}.$$
Hence, part (a) of \propref{prop: contiguity results1} obtains a contiguity equivalence 
$$\begin{aligned}
\overline{f_1} \cdot \overline{g_1} & \simeq
(f_1 \cdot g_1) \circ (\alpha_{m_1+r_1+1}^{\overline{r} - r_1} \times \alpha_{n_1+s_1+1}^{\overline{s} - s_1})
\circ (\alpha_{m_1+r_1+1}^{\overline{m} - m_1} \times \alpha_{n_1+s_1+1}^{\overline{n} - n_1})\\
&\simeq (f_1 \cdot g_1)\circ (\alpha_{m_1+r_1+1}^{\overline{m} +\overline{r} - m_1 - r_1} \times \alpha_{n_1+s_1+1}^{\overline{n} + \overline{s}  - n_1 -s_1}),
\end{aligned}$$
which is a trivial extension $\overline{(f_1 \cdot g_1)}$ of $f_1 \cdot g_1$.  Thus we have an extension-contiguity equivalence $f_1 \cdot g_1 \approx \overline{f_1} \cdot \overline{g_1}$.  Likewise, we obtain an extension-contiguity equivalence $\overline{f_2} \cdot \overline{g_2}  \approx  f_2 \cdot g_2$.

Now the contiguity equivalence $\overline{f_1} \simeq \overline{f_2} \colon I_{\bar{m}} \times I_{\bar{n}} \to X$ may be patched into a contiguity equivalence $$\overline{f_1}\cdot \overline{g_1} \simeq \overline{f_2}\cdot \overline{g_1}\colon I_{\bar{m} + \bar{r} + 1} \times I_{\bar{n}+ \bar{s}+1} \to X$$ as in Lemma~\ref{localized contiguity equiv}.  Similarly, the contiguity equivalence 
$\overline{g_1} \simeq \overline{g_2}$ gives a contiguity equivalence 
$$\overline{g_1}\circ T_{\overline{m}+1, \overline{n}+1} \simeq \overline{g_2}\circ T_{\overline{m}+1, \overline{n}+1} \colon [\bar{m} +1, \bar{m}+\bar{r} + 1] \times [\bar{n} +1, \bar{n}+\bar{s} + 1] \to X,$$
by part (a) of \propref{prop: contiguity results1}. Then we may patch this contiguity equivalence into one
$$\overline{f_2}\cdot \overline{g_1} \simeq \overline{f_2}\cdot \overline{g_2} \colon I_{\bar{m} + \bar{r} + 1} \times I_{\bar{n}+ \bar{s}+1} \to X$$ as in Lemma~\ref{localized contiguity equiv}.
Assembling these parts, we have an extension-contiguity equivalence
$$f_1\cdot g_1 \approx \overline{f_1} \cdot \overline{g_1} \simeq \overline{f_2}\cdot \overline{g_1} \simeq \overline{f_2}\cdot \overline{g_2} \approx f_2\cdot g_2.$$
\end{proof}

By Proposition~\ref{prop: operation well-defined}, we have a well-defined binary operation in $F(X,x_0)$ defined by setting  $[f]\cdot[g] := [f\cdot g]$ for $[f], [g]\in F(X,x_0)$.

\begin{theorem}\label{thm: group} With the binary operation given above, the set of equivalence classes $F(X,x_0)$ is a group.
\end{theorem}

\begin{proof}
Associativity follows immediately since  $(f\cdot g)\cdot h=f\cdot (g\cdot h)$ at the level of maps.

Next, let $c_{x_0}\colon I_{m} \times I_{n}\to X$ be the constant map at $x_0 \in X$ from any rectangle.  Any such map may be viewed as a trivial extension of the constant map $c_{x_0}\colon I_{0,0}\to X$, where $I_{0,0} = \{ (0, 0) \}$.  We show that  $[c_{x_0}]$ acts as a two-sided identity, where---by the preceding remark---we may as well assume the representative $c_{x_0}$ has domain the single point $(0, 0)$.   Let $f\colon \left( I_{m}\times I_{n}, \partial (I_{m}\times I_{n}) \right) \to (X, x_0)$ be any face sphere. Multiplying by our putative identity on the right, we see that  
 $f\cdot c_{x_0}\colon \left( I_{m+1}\times I_{n+1}, \partial (I_{m+1}\times I_{n+1}) \right) \to (X, x_0)$ is equal to the trivial extension $\bar f\colon \left( I_{m+1}\times I_{n+1}, \partial (I_{m+1}\times I_{n+1}) \right)\to (X, x_0)$. Thus $[f]\cdot[c_{x_0}] = [f\cdot c_{x_0}] = [\bar f] = [f]$ and so $[c_{x_0}]$ acts on the right as an identity element.  On the left, we see that $c_{x_0}\cdot f\colon \left( I_{m+1}\times I_{n+1}, \partial (I_{m+1}\times I_{n+1}) \right) \to (X, x_0)$ may be written as the result of doubling the first row and column of $f$, when $f$ is represented array style.  With \corref{cor: different doublings}, we may write
 $$c_{x_0}\cdot f = f\circ (\alpha_0\times\alpha_0) \simeq 
 f\circ (\alpha_m\times\alpha_n) = \overline{f},$$
 with $\overline{f}$ a trivial extension of $f$.  Thus, we have  $[c_{x_0}]\cdot [f] = [f]$ and $[c_{x_0}]$ acts as a left identity too.

Now we consider inverses.  Suppose  $[f] \in F(X,x_0)$ is represented by a face sphere  $f\colon \left( I_{m}\times I_{n}, \partial (I_{m}\times I_{n}) \right) \to (X, x_0)$.  Consider the equivalence class represented by the map $\widetilde{f}\colon \left( I_{m}\times I_{n}, \partial (I_{m}\times I_{n}) \right) \to (X, x_0)$
defined on vertices  by 
 $$\widetilde{f}(i,j):=f(m-i,j).$$
To check that $\widetilde{f}$ is a simplicial map, it is sufficient to check that $\widetilde{f}$ maps $3$-simplices of $I_{m}\times I_{n}$ to simplices of $X$.  Now each $3$-simplex 
$\sigma = \{ (i, j), (i+1, j), (i, j+1), (i+1, j+1) \}$ of $I_{m}\times I_{n}$  has its counterpart $\sigma' = \{ (m-i, j), (m-i-1, j), (m-i, j+1), (m-i-1, j+1) \}$.  We have $\widetilde{f}(\sigma) = f(\sigma')$, which is a simplex of $X$ since $f$ is a simplicial map.

We now claim that  $f \cdot \widetilde{f} \approx c_{x_0}$.  As a pre-processing step,  we have an extension-contiguity equivalence 
$$f \cdot \widetilde{f} \approx (f \mid \widetilde{f}) \colon \left( I_{2m+1}\times I_{n}, \partial (I_{2m+1}\times I_{n}) \right) \to (X, x_0),$$
where we define $(f \mid \widetilde{f})$ as 
$$
(f \mid \widetilde{f})(i, j)=
    \begin{cases}
        f(i,j) & \text{if } 0\leq i \leq m\\
        \widetilde{f}(i-(m+1),j) & \text{if } m+1 \leq i \leq 2m+1.\\
    \end{cases}
$$
This extension-contiguity equivalence may be indicated pictorially as follows:
\[
f\cdot \widetilde{f} =
\vcenter{\hbox{\begin{tikzpicture}[scale=.4]
\draw[thick] (0,0) rectangle (4,6);
\draw[thick] (0,0) rectangle (2,3) node[pos=.5] {$f$};
\draw[thick] (2,3) rectangle (4,6) node[pos=.5] {$\widetilde{f}$};
\end{tikzpicture}
}}
\simeq
\vcenter{\hbox{\begin{tikzpicture}[scale=.4]
\draw[thick] (0,0) rectangle (4,6);
\draw[thick] (2, 0)--(2, 6);
\draw[thick] (0,0) rectangle (2,3) node[pos=.5] {$f$};
\draw[thick] (2,0) rectangle (4,3) node[pos=.5] {$\widetilde{f}$};
\end{tikzpicture}
}}
\approx
\vcenter{\hbox{\begin{tikzpicture}[scale=.4]
\draw[thick] (0,0) rectangle (2,3) node[pos=.5] {$f$};
\draw[thick] (2,0) rectangle (4,3) node[pos=.5] {$\widetilde{f}$};
\end{tikzpicture}
}}
= (f\mid \widetilde{f})
\]
In terms of symbols, we have a contiguity equivalence 
$$\widetilde{f} \circ (\mathrm{id} \times \alpha_0^{n+1})\circ T_{m+1, 0} \colon [m+1, 2m+1]\times I_{2n+1} \to X$$
from \corref{cor: different doublings} and part (a) of \propref{prop: contiguity results1}.  This may be patched into a contiguity equivalence
$$f \cdot \widetilde{f} \simeq (f \mid \widetilde{f}) \circ (\mathrm{id} \times \alpha_n^{n+1})\colon I_{2m+1}\times I_{2n+1} \to X$$
by \lemref{localized contiguity equiv}. Since $(f \mid \widetilde{f}) \circ (\mathrm{id} \times \alpha_n^{n+1})$ is a trivial extension of $(f \mid \widetilde{f})$, we obtain the extension-contiguity equivalence
$$f \cdot \widetilde{f} \simeq (f \mid \widetilde{f}) \circ (\mathrm{id} \times \alpha_n^{n+1}) \approx (f \mid \widetilde{f})$$
pictured above. 
 
Now display the values of $(f\mid \widetilde{f})$ array style on the vertices of $I_{2m+1} \times I_{n}$ in column-wise form as 
$$(f\mid \widetilde{f}) = [\mathbf{x}_0\mid \mathbf{v}_1\mid\mathbf{v}_2\mid \cdots\mid \mathbf{v}_{m-1}\mid \mathbf{x}_0\mid \mathbf{x}_0\mid \mathbf{v}_{m-1}\mid \cdots\mid \mathbf{v}_2\mid \mathbf{v}_1\mid \mathbf{x}_0],$$
where each $\mathbf{v}_i$ is a column vector of $n+1$ entries from $X$ given by
$$\mathbf{v}_i = \left[ \begin{array}{c} f(i, n) = x_0\\  f(i, n-1)\\ \vdots \\ f(i, 2)\\ f(i, 1)\\ f(i, 0) = x_0 \end{array}\right].$$
and $\mathbf{x}_0$ denotes the column vector of the same length each of whose entries is $x_0$.
Since the middle two columns are repeats (of $\mathbf{x}_0$), we may write $(f\mid \widetilde{f}) = g_{m}\circ (\alpha_{m}\times \mathrm{id})$,  where the map 
$g_{m}\colon \left( I_{2m}\times I_{n}, \partial (I_{2m}\times I_{n}) \right) \to (X, x_0)$ may be written in similar column-wise form as 
$$g_{m} = [\mathbf{x}_0\mid \mathbf{v}_1\mid\mathbf{v}_2\mid \cdots\mid \mathbf{v}_{m-1}\mid \mathbf{x}_0 \mid \mathbf{v}_{m-1}\mid \cdots\mid \mathbf{v}_2\mid \mathbf{v}_1\mid \mathbf{x}_0].$$
Namely, we have ``collapsed" the repeated $\mathbf{x}_0$ columns in the center into a single column.  From \corref{cor: different doublings}, we now have an extension-contiguity equivalence
$$(f\mid \widetilde{f}) = g_{m}\circ (\alpha_{m}\times \mathrm{id}) \simeq g_{m}\circ (\alpha_{2m}\times \mathrm{id}) \approx g_{m},$$  
as $g_{m}\circ (\alpha_{2m}\times \mathrm{id})$ is a trivial extension of $g_{m}$.

In similar style, we may  define a map
$g_r \colon \left( I_{2r} \times I_{n}, \partial (I_{2r} \times I_{n})\right) \to (X, x_0)$ for each $r = 1, \ldots, m$  by
$$g_r(i, j) = \begin{cases} f(i, j) & i =0, \ldots, r\\ f(2r-i, j) & i= r+1, \ldots, 2r.\end{cases}$$
The map $g_m$ we arrived at above is the case in which $r = m$, and the general $g_r$ may be pictured in array style and column-wise as a reduced form of $g_m$, with 
$$g_r = [\mathbf{x}_0\mid \mathbf{v}_1\mid\cdots\mid \mathbf{v}_{r-1}\mid \mathbf{v}_r\mid \mathbf{v}_{r-1}\mid \cdots\mid \mathbf{v}_1\mid \mathbf{x}_0].$$
We note in this case that:
\[ g_{r-1}\circ (\alpha_{r-1}^2 \times \mathrm{id}) =  [\mathbf{x}_0\mid \mathbf{v}_1\mid\cdots\mid \mathbf{v}_{r-1}\mid \mathbf{v}_{r-1}\mid \mathbf{v}_{r-1}\mid \cdots\mid \mathbf{v}_1\mid \mathbf{x}_0]. \]

\emph{Claim}. For each $r \in \{m, m-1, \ldots, 1 \}$, we have $g_r \sim g_{r-1}\circ (\alpha_{r-1}^2 \times \mathrm{id}) \approx g_{r-1}$.

\emph{Proof of Claim}. \corref{cor: different doublings} gives an extension-contiguity equivalence $g_{r-1}\circ (\alpha_{r-1}^2 \times \mathrm{id}) \simeq  g_{r-1}\circ (\alpha_{2r-2}^2 \times \mathrm{id}) \approx g_{r-1}$, so we need only show a contiguity   $g_r \sim g_{r-1}\circ (\alpha_{r-1}^2 \times \mathrm{id})$.

Since $g_r$ and $g_{r-1}\circ (\alpha_{r-1}^2 \times \mathrm{id})$ differ only on those vertices of $I_{2r} \times I_{n}$ in column $r$, we need only check the contiguity condition for simplices with a vertex in column $r$.  The 
typical pair of such simplices, and the values the maps take on their vertices, are illustrated in Figures~\ref{fig: g_r vs. g'_r part 1} and \ref{fig: g_r vs. g'_r part 2}.
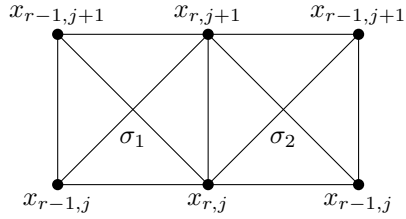
\begin{figure}[h!]
\[
\begin{tikzpicture}[scale=2]
\foreach \x in {0,1,2} {
 \foreach \y in {0,1} {
   \node[inner sep=1.5pt, circle, fill=black]  at (\x,\y) {};
  }
 }
\foreach \x in {0, 1} {
 \foreach \y in {0, 1} {
      \draw (\x,\y) -- (\x+1,\y);
 }
}
\foreach \x in {0, 1, 2} {
      \draw (\x,0) -- (\x,1);
}
\foreach \x in {0,1} {
      \draw (\x,0) -- (\x+1,1);
}
\foreach \x in {0,1} {
      \draw (\x,1) -- (\x+1,0);
}
\node[below]  at (0,0) {$x_{r-1, j}$};
\node[below]  at (1,0) {$x_{r, j}$};
\node[below]  at (2,0) {$x_{r-1, j}$};
\node[above]  at (0,1) {$x_{r-1, j+1}$};
\node[above]  at (1,1) {$x_{r, j+1}$};
\node[above]  at (2,1) {$x_{r-1, j+1}$};
\node at (0.5,0.3) {$\sigma_1$};
\node at (1.5,0.3) {$\sigma_2$};
\end{tikzpicture}
\]
\caption{Values of $g_{r}$ on the vertices of $[r-1, r+1]\times [j, j+1]$, with $x_{i, j}$ denoting $f(i, j)$. }\label{fig: g_r vs. g'_r part 1}
\end{figure}
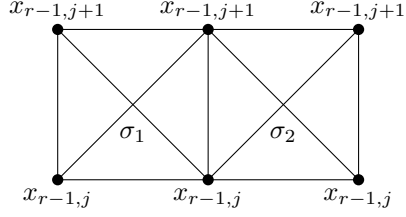
\begin{figure}[h!]
\[
\begin{tikzpicture}[scale=2]
\foreach \x in {0,1,2} {
 \foreach \y in {0,1} {
   \node[inner sep=1.5pt, circle, fill=black]  at (\x,\y) {};
  }
 }
\foreach \x in {0, 1} {
 \foreach \y in {0, 1} {
      \draw (\x,\y) -- (\x+1,\y);
 }
}
\foreach \x in {0, 1, 2} {
      \draw (\x,0) -- (\x,1);
}
\foreach \x in {0,1} {
      \draw (\x,0) -- (\x+1,1);
}
\foreach \x in {0,1} {
      \draw (\x,1) -- (\x+1,0);
}
\node[below]  at (0,0) {$x_{r-1, j}$};
\node[below]  at (1,0) {$x_{r-1, j}$};
\node[below]  at (2,0) {$x_{r-1, j}$};
\node[above]  at (0,1) {$x_{r-1, j+1}$};
\node[above]  at (1,1) {$x_{r-1, j+1}$};
\node[above]  at (2,1) {$x_{r-1, j+1}$};
\node at (0.5,0.3) {$\sigma_1$};
\node at (1.5,0.3) {$\sigma_2$};
\end{tikzpicture}
\]
\caption{Values of $g_{r-1}\circ (\alpha_{r-1}^2 \times \mathrm{id})$ on the vertices of $[r-1, r+1]\times [j, j+1]$, with $x_{i, j}$ denoting $f(i, j)$. }\label{fig: g_r vs. g'_r part 2}
\end{figure}
We check the contiguity property directly.  On the $3$-simplex of $I_{2r} \times I_n$ on the left in both Figures, namely $\sigma_1 = \{ (r-1, j), (r, j), (r-1, j+1), (r, j+1) \}$,  we have 
$$g_{r}(\sigma_1) \cup g_{r-1}\circ (\alpha_{r-1}^2 \times \mathrm{id})(\sigma_1)  = \{ x_{r-1, j}, x_{r-1, j+1}, x_{r, j}, x_{r, j+1} \}  = f(\sigma_1),$$
which is a simplex of $X$ as $f$ is a simplicial map.  (As in the Figures, here we have used $x_{i, j}$ to denote $f(i, j)$.  When we write $f(\sigma_1)$, we mean $f$ in its original version as applied to $I_m \times I_n$.) Similarly, we find that  
$$g_{r}(\sigma_2) \cup g_{r-1}\circ (\alpha_{r-1}^2 \times \mathrm{id})(\sigma_2)  = f(\sigma_1).$$
It follows that $g_{r} \sim g_{r-1}\circ (\alpha_{r-1}^2 \times \mathrm{id})$.
\emph{End of Proof of Claim}.

Finally, an easy induction using the previous claim, picking up from where we left off before the claim, gives a sequence of extension-contiguity equivalences
$$f\cdot \widetilde{f} \approx(f\mid \widetilde{f}) \approx g_{m} \approx \cdots \approx g_{0},$$
where this last map is a constant map.  This completes the proof that $[\widetilde{f}]$ is a right inverse for $[f] \in F(X, x_0)$.  As the operation is associative, this is sufficient to show that $[\widetilde{f}]$ is a two-sided inverse, and thus that $F(X, x_0)$ is a group.
\end{proof}

Later in the paper (\thmref{thm: Face group iso homotopy group}) we will show that $F(X, x_0)$ is isomorphic to $\pi_2(|X|)$, the second homotopy group of the spatial realization of $X$.  Since the second homotopy group is an abelian group, it will follow that our $F(X, x_0)$ is an abelian group.  But we may show this directly here with a simple combinatorial argument.

\begin{theorem}\label{thm: Face group abelian}
Let $(X, x_0)$ be any based simplicial complex. The face group $F(X, x_0)$ is an abelian group. 
\end{theorem}    

\begin{proof}
Suppose face spheres $f\colon \left(I_{m} \times I_{n}, \partial (I_{m} \times I_{n}) \right) \to (X,x_0)$ and $g\colon \left(I_{r} \times I_{s}, \partial (I_{r} \times I_{s}) \right) \to (X,x_0)$ represent elements $[f]$ and $[g]$ of $F(X, x_0)$.  We show a contiguity equivalence $f\cdot g \simeq g\cdot f \colon \left(I_{m+r+1} \times I_{n+s+1}, \partial (I_{m+r+1} \times I_{n+s+1}) \right) \to (X,x_0)$.
This contiguity equivalence may be indicated pictorially as a sequence of contiguity equivalences follows:
\[
f\cdot g =
\vcenter{\hbox{\begin{tikzpicture}[scale=.2]
\draw[thick] (0,0) rectangle (7,7);
\draw[thick] (0,0) rectangle (2,3) node[pos=.5] {$f$};
\draw[thick] (2,3) rectangle (7,7) node[pos=.5] {$g$};
\end{tikzpicture}
}}
\simeq
\vcenter{\hbox{\begin{tikzpicture}[scale=.2]
\draw[thick] (0,0) rectangle (7,7);
\draw[thick] (2, 4)--(2, 7);
\draw[thick] (0,0) rectangle (2,3) node[pos=.5] {$f$};
\draw[thick] (2,0) rectangle (7,4) node[pos=.5] {$g$};
\end{tikzpicture}
}}
\simeq
\vcenter{\hbox{\begin{tikzpicture}[scale=.2]
\draw[thick] (0,0) rectangle (7,7);
\draw[thick] (0,4) rectangle (2,7) node[pos=.5] {$f$};
\draw[thick] (2,0) rectangle (7,4) node[pos=.5] {$g$};
\end{tikzpicture}
}}
\simeq
\vcenter{\hbox{\begin{tikzpicture}[scale=.2]
\draw[thick] (0,0) rectangle (7,7);
\draw[thick] (0, 4)--(2, 4);
\draw[thick] (5,4) rectangle (7,7) node[pos=.5] {$f$};
\draw[thick] (2,0) rectangle (7,4) node[pos=.5] {$g$};
\end{tikzpicture}
}}
\simeq
\vcenter{\hbox{\begin{tikzpicture}[scale=.2]
\draw[thick] (0,0) rectangle (7,7);
\draw[thick] (5,4) rectangle (7,7) node[pos=.5] {$f$};
\draw[thick] (0,0) rectangle (5,4) node[pos=.5] {$g$};
\end{tikzpicture}
}}
= g \cdot f
\]
In symbolic terms, the first of these contiguity equivalences may be given as follows.  We have a contiguity equivalence 
$$g \circ (\mathrm{id} \times \alpha_0^{n+1})\circ T_{m+1, 0} \simeq g \circ (\mathrm{id} \times \alpha_n^{n+1})\circ T_{m+1, 0}\colon [m+1, m+r+1]\times I_{n+s+1} \to X$$
from \corref{cor: different doublings} and part (a) of \propref{prop: contiguity results1}.  This may be patched into a contiguity equivalence
$$f \cdot g \simeq h \colon I_{m+r+1}\times I_{n+s+1} \to X$$
by \lemref{localized contiguity equiv}, where $h$ is the map given on vertices by 
$$
h(i,j)=
    \begin{cases}
        f(i,j) & \text{if } (i,j)\in I_m\times I_n\\
        g\circ T_{m+1, 0}(i, j) & \text{if } (i,j)\in [m+1,m+r+1]\times I_s\\
        x_0 & \text{otherwise}
    \end{cases}
$$
and pictured in the figure second from left. The remaining contiguity equivalences follow similarly, working locally in $I_m \times I_{n+s+1}$ for the next, then working locally in $I_{m+r+1} \times [s+1, s+n+1]$ for the third, then in $I_{m+n+1}\times I_s$ for the fourth, final one.  We omit the details.  
\end{proof}

The face group is well-behaved functorially, as we now show. Suppose we have a simplicial map $\phi\colon X \to Y$ and $\phi(x_0) = y_0$.  Then, if $f\colon \left(I_{m} \times I_{n}, \partial (I_{m} \times I_{n}) \right) \to (X,x_0)$ is a face sphere in $(X, x_0)$, the composition $\phi\circ f \colon \left(I_{m} \times I_{n}, \partial (I_{m} \times I_{n}) \right) \to (Y,y_0)$ is a face sphere in $(Y, y_0)$.  Furthermore, it follows directly from the definitions that we have agreement of vertex maps
$$\phi\circ (f\cdot g) = (\phi\circ f)\cdot (\phi\circ g) \colon \left(I_{m} \times I_{n}, \partial (I_{m} \times I_{n}) \right) \to (Y,y_0),$$
where the ``$\cdot$" on the left of the equals sign denotes operation on face spheres in $(X, x_0)$ and that on the right of the equals sign denotes operation on face spheres in $(Y, y_0)$. 
\begin{definition}
Given a simplicial map $\phi\colon X \to Y$ with $\phi(x_0) = y_0$.  We define a map
$$\phi_\# \colon F(X, x_0) \to F(Y, y_0)$$
by setting $\phi_\#( [f]) := [\phi\circ f]$.  By the preceding comments, this defines a homomorphism of face groups.  We call this homomorphism the \emph{induced homomorphism of face groups} (induced by $\phi$).
\end{definition}

The functorial nature of the face group is included in the next result.

\begin{proposition}
With the above notation and definitions, we have:
\begin{itemize}
\item[(a)] If $\mathrm{id} \colon X \to X$ denotes the identity map of $X$, then $\mathrm{id}_\# \colon F(X, x_0) \to F(X, x_0)$ is the identity homomorphism.   
\item[(b)] Suppose we have based simplicial maps $\phi \colon (X, x_0) \to (Y, y_0)$ and $\psi \colon (Y, y_0) \to (Z, z_0)$.  Then $(\psi\circ \phi)_\# =  (\psi)_\# \circ (\phi)_\#$. 
\item[(c)] If we have contiguity equivalent maps $\phi \simeq \psi \colon \colon (X, x_0) \to (Y, y_0)$, then $\phi_\# = \psi_\# \colon F(X, x_0) \to F(Y, y_0)$   
\end{itemize}
\end{proposition}

\begin{proof}
Parts (a) and (b) follow immediately from the definitions.  Part (c)  follows from part (a) of  \propref{prop: contiguity results1}. 
\end{proof}

Recall from \secref{sec: simplicial basics} that $X \times Y$ denotes the categorical product of simplicial complexes $X$ and $Y$.  
 
\begin{theorem}\label{thm: face group product}
Given based simplicial complexes $(X, x_0)$ and $(Y, y_0)$, we have an isomorphism of face groups
$$F\left(X \times Y, (x_0, y_0)\right)   \cong F(X, x_0) \times F(Y, y_0).$$
The product on the right denotes direct sum of abelian groups.  
\end{theorem} 

\begin{proof}
The projections onto the coordinates $p_1 \colon X \times Y \to X$ and $p_2 \colon X \times Y \to Y$ 
 induce homomorphisms of Face groups and we may take them to be the coordinate homomorphisms of a homomorphism
$$\Psi:= \left( (p_1)_\#, (p_2)_\# \right) \colon F\left(X \times Y, (x_0, y_0)\right) \to F(X, x_0) \times F(Y, y_0).$$
We check that $\Psi$ is surjective and injective. Suppose we have $( [f_1], [f_2])  \in F(X, x_0) \times F(Y, y_0)$ with $f_1$ a face sphere in $X$ and $f_2$ a face sphere in $Y$.  If $f_1 \colon I_m \times I_n \to X$, then let $c_{y_0} \colon I_m \times I_n \to Y$ be the constant map defined on the same-sized rectangle. Then $(f_1, c_{y_0})$ is a  face sphere in $X \times Y$. Similarly, we have $( c_{x_0}, f_2)$ a face sphere in $X \times Y$. We calculate as follows: 
$$\begin{aligned}
\Psi\left(  [(f_1, c_{y_0})]  \right) &=    \left( (p_1)_\#( [(f_1, c_{y_0})],  (p_2)_\#( [(f_1, c_{y_0})] )\right) \\
&=  \left( (p_1)_\#( [(f_1, c_{y_0})]],  (p_2)_\#( [(f_1, c_{y_0})]) \right) \\
&=  \left(  [p_1 \circ (f_1, c_{y_0})],  [p_2\circ (f_1, c_{y_0})]  \right) \\
&=  \left(  [f_1],  [c_{y_0}] \right).
\end{aligned}
$$
A similar calculation shows that we have $\Psi\left(  [(c_{x_0}, f_2)]  \right) =   \left( [c_{x_0}],  [f_2] \right)$.  Then 
$$\begin{aligned}
\Psi\left(  [(f_1, c_{y_0})] \cdot   [(c_{x_0}, f_2)] \right) &= \Psi\left(  [(f_1, c_{y_0})] \right) \cdot \Psi\left(   [(c_{x_0}, f_2)] \right) =  \left(  [f_1],  [c_{y_0}] \right)\cdot \left( [c_{x_0}],  [f_2] \right) \\
&=  \left(  [f_1] \cdot [c_{x_0}],  [c_{y_0}] \cdot  [f_2] \right) =  \left(  [f_1],  [f_2] \right).
\end{aligned}$$
It follows that $\Psi$ is a surjection. 

Now suppose that $f$ is a face sphere in $X \times Y$ and we have $\Psi( [f] ) = ( [c_{x_0}], [c_{y_0}] )$.  We have $f = (f_1, f_2)$ with $f_1$ a face sphere in $X$ and $f_2$ a face sphere in $Y$ (both mapping from the same rectangle $I_m \times I_n$, say.   Then $(p_1)_\#([f])  = [p_1\circ(f_1, f_2)] = [f_1] = [c_{x_0}]$, by assumption.  This means we have $f_1 \approx c_{x_0}$, so there is some trivial extension and a contiguity equivalence 
$$f_1 \circ (\alpha^{\overline{m} - m}_m, \alpha^{\overline{n} - n}_n) \simeq c_{x_0}.$$

We now have some trivial extension of $f$ that satisfies
$$f \circ (\alpha^{\overline{m} - m}_m \times \alpha^{\overline{n} - n}_n) = \left( f_1 \circ (\alpha^{\overline{m} - m}_m \times \alpha^{\overline{n} - n}_n), \overline{f_2} \right) 
\simeq \left( c_{x_0}, \overline{f_2} \right),$$
with the contiguity equivalence following from part (c) of \propref{prop: contiguity results1} .  Turning to the second coordinate, we are also assuming that we have $(p_2)_\#([f])  = [p_2\circ(f_1, f_2)] = [f_2] = [c_{y_0}]$. The extension-contiguity equivalence $f_2 \approx c_{y_0}$ implies a contiguity equivalence between some trivial extension of $\overline{f_2}$ and some constant map $c_{y_0}$, where $\overline{f_2}$ is the trivial extension of $f_2$ from the first step.  Arguing as we did for the first coordinate results in an extension-contiguity equivalence $f \approx (c_{x_0}, c_{y_0})$.  It follows that $\Psi$ is also an injection.    
\end{proof}
 
\section{An ``Intrinsic" Description of the Face Group}\label{sec: combinatorial face group}

We may describe $F(X, x_0)$ in a way that positions it as a very natural higher-dimensional version of the edge group of a simplicial complex. The edge group of a simplicial complex is defined in terms of edge loops: sequences of vertices of $X$.  This formulation of the edge group is appealing as it has an algorithmic aspect to it that is lacking, say, in the definition of the fundamental group of a topological space.    But these edge loops transparently correspond to simplicial maps of the form $I_m \to X$ for various $m$. Our face group, meanwhile, is defined in terms of simplicial maps of the form $I_m \times I_n \to X$. But we can ``reverse engineer" this formulation, again in a very transparent way, to one that corresponds directly to the ``sequence of vertices" formulation of the edge group.     We must allow for some complications, since we are describing a higher-dimensional homotopy group.  First, recall the description of the edge group.

For $(X, x_0)$ a (based) simplicial complex, say an \emph{edge loop} (of length $m$) is a sequence $( v_0, v_1, \ldots, v_m )$ of vertices of $X$ with $v_0 = v_m = x_0$ and, for each $i=0, \ldots, m-1$, adjacent pairs $\{ v_i, v_{i+1}\}$ an edge ($1$-simplex) of $X$.  Then the edge group of $X$ is the set of equivalence classes of all edge loops (all lengths) under the  equivalence relation generated by the two types of (reflexive, symmetric) move:

(i)  $( v_0, \ldots, v_i, \ldots, v_m ) \sim ( v_0, \ldots, v_i, v_i, \ldots, v_m )$ (i.e., repeat a vertex or delete a repeated vertex);

(ii)   $( v_0, \ldots, v_{i-1}, v_i, v_{i+1}, \ldots, v_m ) \sim ( v_0, \ldots, v_{i-1}, v'_{i}, v_{i+1}, \ldots, v_m )$ when  $\{ v_{i-1}, v_i, v'_i \}$ and $\{ v_i, v'_i,  v_{i+1} \}$ are both simplices of $X$.  

Now this is not precisely the form of the generating relations given in \cite{Mau96}, for example, but is easily seen to be an equivalent phrasing.  In fact, this is nothing other than a group that may be described as a one-dimensional version of our face group, with the edge loops corresponding to simplicial maps $(I_m, \partial I_m) \to (X, x_0)$ and the equivalence relation on such maps being a suitable version of extension-contiguity equivalence.  This follows, since move (i) is transparently an extension and move (ii) is a contiguity (we check on the $1$-simplices).   So, we may formulate our face group in similar terms (a sort of ``intrinsic" formulation that involves only internal data concerned with $X$, rather than in terms of maps from the rectangle into $X$).

Suppose we have  a map $f \colon \left( I_{m} \times I_{n}, \partial (I_{m} \times I_{n}) \right) \to (X, x_0)$ that represents an equivalence class in the face group of $X$.  Restricting to each row of vertices gives a map
$$f (-, j)  \colon \left( I_{m} \times \{j\}, \{ (0, j), (m, j) \} \right) \to (X, x_0)$$
that is an edge loop (of length $m$) in $X$.  So, reading across the rows, working row-by-row from the bottom row ($j=0$) up to the top row ($j=n$), we may interpret $f$ as a sequence of $n+1$ edge loops that together satisfy certain coherence conditions as we move from one row to the next (between adjacent edge loops).  One way of expressing all this is as follows.     

\begin{definition}[Intrinsic Face Group]
By a \emph{face sphere (bis)} in $X$ (of size $m \times n$) we mean a sequence $( l^0, l^1, \ldots, l^n )$ of edge loops of length $m$ in $X$ that satisfy the following conditions (in which we write $l^j = ( v^j_0, \ldots, v^j_m )$ for each $j = 0, \ldots, n$):
\begin{itemize}
\item[(a)] Both $l^0$ and $l^n$ are the constant loop at $x_0$.  That is, we have $v^0_i = v^n_i = x_0$ for $i = 0, \ldots, m$. (Note that we also have $v^j_0 = v^j_m = x_0$ for $j = 0, \ldots, n$ as we assume that each row is an edge loop in $X$.)
\item[(b)] for each $j = 0, \ldots, n-1$ and  each $i = 0, \ldots, m-1$,  $\{ v^j_i, v^j_{i+1}, v^{j+1}_{i}, v^{j+1}_{i+1} \}$ is a  simplex of $X$; 
\end{itemize}
Then we define an equivalence relation on the set of all face spheres (of all dimensions) in $X$ as the one generated by the following types of (reflexive, symmetric)  move: 
\begin{itemize}
\item[(i)] $( l^0, \ldots,  l^j, \ldots, l^n ) \sim ( l^0, \ldots,  l^j, l^j, \ldots, l^n )$;
\item[(ii)] for each $j = 0, \ldots, n$, set $\widehat{l^j} = ( v^j_0, \ldots, v^j_i, v^j_i, \ldots, v^j_m )$ for some $i$ (same $i$ for all $j$), then  $( l^0, \ldots,  l^n ) \sim ( \widehat{l^0}, \ldots, \widehat{l^n} )$;
\item[(iii)] for some $1 \leq i \leq m-1$ and $1 \leq j \leq n-1$, suppose we have $l^j =  ( v^j_0, \ldots, v^j_{i-1}, v^j_i, v^j_{i-1}, \ldots,  v^j_m )$ and $\widetilde{l}^j =  ( v^j_0, \ldots, v^j_{i-1}, \widetilde{v}^j_i, v^j_{i-1}, \ldots,  v^j_m )$.  That is, the edge loops $l^j$ and $\widetilde{l}^j $ differ only in their $i$th entry.  Then we have $( l^0, \ldots, l^{j-1}, l^j, l^{j+1}, \ldots,  l^n ) \sim ( l^0, \ldots, l^{j-1}, \widetilde{l}^j, l^{j+1}, \ldots, l^n )$ when each of the following $4$ sets is a simplex of $X$:
$$\begin{aligned}
\{ v^{j-1}_{i-1}, v^{j-1}_{i}, v^{j}_{i-1}, v^{j}_{i}, \widetilde{v}^{j}_{i} \} \qquad & \qquad  \{  v^{j-1}_{i}, v^{j-1}_{i+1}, v^{j}_{i}, v^{j}_{i+1}, \widetilde{v}^{j}_{i} \} \\
\{ v^{j}_{i-1}, v^{j}_{i}, v^{j+1}_{i-1}, v^{j+1}_{i}, \widetilde{v}^{j}_{i} \} \qquad & \qquad  \{  v^{j}_{i}, v^{j}_{i+1}, v^{j+1}_{i}, v^{j+1}_{i+1}, \widetilde{v}^{j}_{i} \} \\
\end{aligned}$$

\end{itemize}
\end{definition}

In the above moves, (i) and (ii) correspond to repetition of a row or column, as discussed in \secref{sec: extn contiguity}.  Move (iii) amounts to changing the $(i, j)$th entry in the array-style matrix of values of a face sphere.  The conditions that must be met correspond to the contiguity condition on the four $3$-simplices of $I_{m} \times I_{n}$ that have $(i,j)$ as a vertex.  When we translate the change in value into maps, they correspond to the maps before and after the change being contiguous.

\section{Simplicial Approximation; Maps $(I_{m, n} , \partial I_{m, n}) \to (X, x_0)$}\label{sec: simp approx}  

We mostly follow the discussion of simplicial approximation from \cite{Bre}.  For certain details, especially for homotopy of simplicial approximations, \cite{HiWi60} is a useful source. In what follows we review some standard material on simplicial approximation. We include proofs for the sake of completeness; they are short in any case. But our treatment here departs from a standard treatment as we do not use barycentric subdivision.  Rather, we develop our own versions of the standard results in the form we want them for the main result  (\thmref{thm: Face group iso homotopy group}).

Let $I_{m,n}$ denote the simplicial complex given by the unit square $I^2$ subdivided evenly into $m\times n$ sub-rectangles (each of dimensions $1/m \times 1/n$) and triangulated by dividing each sub-rectangle into two $2$-simplices by the diagonal from bottom left to top right.   An illustration is given in
Figure~\ref{fig: square triangulation}.  
\begin{figure}[h!]
\[
\begin{tikzpicture}[scale=2]
\foreach \x in {0, 0.2, 0.4, 0.6, 0.8, 1} {
 \foreach \y in {0, 0.25, 0.5, 0.75, 1} {
   \node[inner sep=1.5pt, circle, fill=black]  at (\x, \y) {};
  }
 }
\foreach \x in {0, 0.2, 0.4, 0.6, 0.8, 1} {
\draw (\x,0) -- (\x,1);
 }
 \foreach \y in {0, 0.25, 0.5, 0.75, 1} {
\draw (0,\y) -- (1,\y);
 }
\foreach \i in {0, 0.2, 0.4, 0.6, 0.8} {
 \foreach \j in {0, 0.25, 0.5, 0.75} {
    \draw (\i,\j) -- (\i+0.2,\j+0.25);
  }
  }
\end{tikzpicture}
\]
\caption{$I^2$ triangulated as  $I_{5, 4}$\label{fig: square triangulation}}
\end{figure}
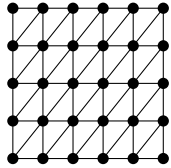
The simplicial complex $I_{m, n}$ is the \emph{Cartesian product} of the simplicial complexes $I_m$ and $I_n$ (each with their natural ordering of vertices),  It is combinatorially different from, and does not have the categorical properties of, our categorical product $I_m \times I_n$.   

Thus, $I_{m,n}$ is a $2$-dimensional simplicial complex and we have $|I_{m,n}| = I^2$. For a point (generally not a vertex) $y \in |K|$ with $K$ any simplicial complex, \emph{the carrier of $y$} is the smallest (closed) simplex of $|K|$ that contains $y$.  Denote the carrier of $y$ by $\mathrm{carr}(y)$  (or by $\mathrm{carr}_K(y)$ if we need to emphasize the $K$).   If $g \colon I_{m,n} \to X$ is a simplicial map, then we denote by $|g| \colon I^2 \to |X|$ its spatial realization (a continuous map).

Again, suppose $X$ is an (abstract)  simplicial complex  with vertices $\{x_i\}$ and a distinguished choice of vertex $x_0$ as basepoint.  Let $f \colon (I^2, \partial I^2) \to (|X|, x_0)$ be a continuous map.  Then a simplicial map $g \colon (I_{m,n}, \partial I_{m,n}) \to (X, x_0)$ is \emph{a simplicial approximation to $f$} if $|g|(y) \in \mathrm{carr}\big( f(y) \big)$ for each $y \in I^2 = |I_{m, n}|$. 

\begin{remark}\label{rem: simplicial approximation to self}
 If $g \colon (I_{m,n}, \partial I_{m,n}) \to (X, x_0)$ is a simplicial map, then $g$ is a simplicial approximation to its own spatial realization $|g| \colon (I^2, \partial I^2) \to (|X|, x_0)$.
\end{remark}

\begin{proposition}
If $g \colon (I_{m,n}, \partial I_{m,n}) \to (X, x_0)$ is a simplicial approximation to $f \colon (I^2, \partial I^2) \to (|X|, x_0)$, then $|g| \simeq f$ (homotopic relative the boundary, as continuous maps).
\end{proposition}

\begin{proof}
Use the straight-line homotopy in each simplex. This homotopy pieces together well into a continuous map since simplices overlap in simplices and so the homotopy will agree on overlaps.  The homotopy will be stationary at points on the boundary $\partial I^2$. 
\end{proof}

For a vertex $v \in K$ any simplicial complex, define $\mathrm{st}(v)$ (or $\mathrm{st}_K(v)$ if we need to emphasize the $K$), \emph{the (open) star of $v$ (in $|K|$)}, as
$$\mathrm{st}(v) := \{ y \in |K| \mid v \in \mathrm{carr}(y) \}.$$

\begin{proposition}[{\cite[Prop.22.8]{Bre}}]\label{prop: star cover}
For $\{ x_i \}$ the vertices of $X$, the sets $\{ \mathrm{st}(x_i) \}$ form an open cover of $|X|$.
\end{proposition} 

\begin{proof}
The stars of vertices cover since, for any $y \in |X|$, we have $y \in \mathrm{st}(v)$ for any vertex $v$ of $\mathrm{carr}(y)$.  Each $\mathrm{st}(v)$ is open because $|K| - \mathrm{st}(v)$ is a union of (closed) simplices and hence closed.  To see this latter point, note that if we assume $y \not\in \mathrm{st}(v)$, then in fact $\mathrm{st}(v)$ and $\mathrm{carr}(y)$ must be disjoint.  As $y \in \mathrm{carr}(y)$, we may write, for each vertex $v$, 
$$|K| - \mathrm{st}(v) = \bigcup_{y \not\in \mathrm{st}(v)} \mathrm{carr}(y).$$
\end{proof}

The next result is our own version of a standard result in the development of simplicial approximation.  We state and prove it here, since it is this particular (and non-standard) form that we need for our development.  

\begin{theorem}[Simplicial Approximation]\label{thm: simp approx}
Let $f \colon (I^2, \partial I^2) \to (|X|, x_0)$ be a continuous map (of pairs).  For some suitably large $m$, there is a simplicial map $g \colon (I_{m,m}, \partial I_{m,m}) \to (X, x_0)$ that is a simplicial approximation to $f$ (and hence $f$ and $|g|$ are homotopic relative the boundary as continuous maps $f \simeq |g|\colon (I^2, \partial I^2) \to (|X|, x_0)$).
\end{theorem}

\begin{proof}
From Proposition~\ref{prop: star cover}, $\{ f^{-1}\big(\mathrm{st}(x_i)\big) \}_{x_i \in V(X)}$---the inverse images of stars of vertices of $X$---form an open cover of $I^2$.  Since $I^2$ is a compact metric space, the Lebesgue covering lemma applies and there is some Lebesgue number $\delta$ for this covering of $I^2$.  On the other hand, for each vertex $a \in I_{m,m}$, we may observe that 
$\mathrm{diam}\big( \mathrm{st}_{I_{m,m}}(a)\big) \leq 2\sqrt{2}/m$ (with equality whenever $a$ is not a vertex on the boundary).  So, let $m$ be large enough so that $2\sqrt{2}/m < \delta$.  Then, by the Lebesgue lemma, each of the subsets  $\mathrm{st}_{I_{m,m}}(a)$  of $I_{m,m}$, when $a$ is a vertex of $I_{m,m}$, is contained in at least one of the open sets  $f^{-1}\big(\mathrm{st}(x_i)\big)$.  

On vertices $a$ of $I_{m,m}$ contained in the boundary $\partial I_{m,m}$, define $g(a) = x_0$.  For every other vertex $a$ of $I_{m,m}$, we have that $f\big(\mathrm{st}_{I_{m,m}}(a)\big)$ is contained in at least one open set  $\mathrm{st}(x_i)$.  For each vertex $a$ of $I_{m,m}$ not contained in the boundary choose one such $x_i$ and define $g(a) = x_i$.  Subject to this choice, this defines $g$ on vertices of $I_{m,m}$.  We now check that any such choice of $g$ extends to a simplicial map  $g \colon (I_{m,m}, \partial I_{m,m}) \to (X, x_0)$ and gives a simplicial approximation to $f$.  

\emph{$g$ is a simplicial map}: It is sufficient to check that $g(\sigma)$ is a simplex of $X$ for each $\sigma = \{ a_1, a_2, a_3 \}$ a $2$-simplex of $I_{m, m}$.    From the definition of $g$, we have $f\big( \mathrm{st}(a_i) \big) \subseteq \mathrm{st}\big( g(a_i) \big)$ for each $i$.    Now, the intersection  $\mathrm{st}(a_1) \cap \mathrm{st}(a_2) \cap \mathrm{st}(a_3)$ is non-empty---indeed, it consists of the (closed) simplex $\sigma$.  Thus, the inclusions of sets 
$$f \big( \bigcap_{i=1, 2, 3}  \mathrm{st}(a_i) \big)  \subseteq  \bigcap_{i=1, 2, 3}  f \big(  \mathrm{st}(a_i) \big) \subseteq  \bigcap_{i=1, 2, 3}   \mathrm{st}\big( g(a_i) \big)$$
give that the intersection $\mathrm{st}\big( g(a_1) \big) \cap \mathrm{st}\big( g(a_2) \big) \cap \mathrm{st}\big( g(a_3) \big)$ is non-empty.  From the definitions, this implies that each $g(a_i) \in \mathrm{carr}(x)$ for any point $x$ in the non-empty intersection. Since $\mathrm{carr}(x)$ is a simplex of $|X|$ and each $g(a_i)$ is a vertex of $X$, it follows that $g(\sigma) = \{ g(a_1), g(a_2), g(a_3) \}$ is a simplex of $X$ and thus $g$ is a simplicial map. 

\emph{$g$ is a simplicial approximation to $f$}: We must show that $|g|(y) \in \mathrm{carr}\big(f(y)\big)$ for each $y \in I^2$.  Suppose that $\mathrm{carr}(y) = \sigma$, a simplex of $I_{m, m}$.  For each vertex $a$ of $\sigma$, we have $a \in  \mathrm{carr}(y)$ or $y \in \mathrm{st}(a)$. Then $f(y) \in f\big( \mathrm{st}(a) \big) \subseteq \mathrm{st} \big( g(a) \big)$ from the definition of $g$ on vertices.  Now  $f(y) \in  \mathrm{st} \big( g(a) \big)$ implies that $g(a) \in \mathrm{carr}\big(f(y)\big)$.  But $\mathrm{carr}\big(f(y)\big)$ is a simplex of $X$ and it follows that $g(\sigma) \subseteq \mathrm{carr}\big(f(y)\big)$.  In particular, we have $|g|(y) \in \mathrm{carr}\big(f(y)\big)$.  
\end{proof}

\begin{lemma}\label{lem: same map contig}
If two simplicial maps $g_1, g_2 \colon (I_{m, m}, \partial I_{m,m}) \to (X, x_0)$ are simplicial approximations to the same continuous map $f  \colon (I^2, \partial I^2) \to (|X|, x_0)$, then $g_1$ and $g_2$ are contiguous.
\end{lemma}

\begin{proof} For suppose that $\sigma$ is a simplex of $I_{m, m}$.  We may write $\sigma = \mathrm{carr}(y)$ for suitable $y$.  Then, since each $g_i$ is a simplicial map and approximates $f$, we have
$$g_i(\sigma) = g_i\big( \mathrm{carr}(y) \big) =   \mathrm{carr}\big(|g_i|(y) \big) \subseteq  \mathrm{carr}\big(f(y) \big).$$
The middle equality here is explained (for any simplicial map) in \cite[Lem.22.5]{Bre} and the last containment follows directly since each $g_i$ approximates $f$.  But  $\mathrm{carr}\big(f(y) \big)$ is some simplex of $X$. It follows that $g_1(\sigma) \cup g_2(\sigma)$ must be a simplex and hence we have $g_1 \sim g_2$.
\end{proof}

To discuss simplicial approximation and homotopy in the way we want to, we introduce a \emph{subdivision map}. Notice that this is a departure from the use of barycentric subdivision as in a standard development of simplicial approximation.

\begin{definition}\label{def: subdivision}
For $I_{m, n}$ a triangulation of $I^2$, and any $k \geq 2$, define a simplicial map
$$\rho_k \colon I_{km, kn} \to I_{m, n}$$
as that induced by the vertex map given as follows: For $(a/km, b/kn)$ a vertex of $I_{km, kn}$, write $a = ki+r$ and $b = kj+s$ with $0 \leq r, s < k$.  Then set
$$\rho_k(a/km, b/kn) = \rho_k((ki+r)/km, (kj+s)/kn) = (i/m, j/n).$$
\end{definition} 

\begin{remark}
For each $k$, the triangulation $I_{km, kn}$ may be regarded as a subdivision of $I_{m, n}$ in the sense that the vertices of $I_{m, n}$ are amongst those of $I_{km, kn}$. The same is not true, for example, of the two triangulations  $I_{m+1, n}$ and $I_{m, n}$ of $I^2$---the vertices of $I_{m+1, n}$ do not contain those of $I_{m, n}$. For some purposes, the ``finer" triangulations obtained by taking $m$ and $n$ larger in $I_{m, n}$ may be used to the same effect as barycentric subdivision in the usual discussion of simplicial approximation.  For our immediate purposes, though, we do actually want to work with the particular triangulations of Definition~\ref{def: subdivision}.     
\end{remark}

\begin{lemma}\label{lem: rho approx id}
Each subdivision map $\rho_k \colon I_{km, kn} \to I_{m, n}$ is a simplicial approximation of the (continuous) identity map $\mathrm{id}\colon I^2 \to I^2$.
\end{lemma}

\begin{proof}
For $y \in I^2$ (considered as the spatial realization of $I_{km, kn}$), we want to show that $|\rho_k|(y) \in \mathrm{carr}_{I_{m, n}}(y)$.  Suppose that 
$$y \in [i/m, (i+1)/m] \times [j/n, (j+1)/n] \subseteq I^2.$$
If $y$ is a point in the interior (topologically speaking) of $[i/m, (i+1)/m] \times [j/n, (j+1)/n]$ and above the diagonal from $(i/m, j/n)$ to $\left( (i+1)/m, (j+1)/n\right)$, then  $\mathrm{carr}_{I_{m, n}}(y)$ is the triangle (including edges and vertices) with vertices 
\begin{equation}\label{eq: upper left}
\{   (i/m, j/n), \left( (i+1)/m, (j+1)/n\right), ( i/m, (j+1)/n) \}.
\end{equation}
For such a $y$, $|\rho_k|$ maps $y$ to a point that is either in the interior of this triangle or on its boundary. If $y$ is a point on the diagonal from $(i/m, j/n)$ to $\left( (i+1)/m, (j+1)/n\right)$, then  $\mathrm{carr}_{I_{m, n}}(y)$ is the diagonal itself (including the endpoints) and $|\rho_k|$ maps such a $y$ to a point that is again on this diagonal.   If $y$ is any point on the left-hand edge from $(i/m, j/n)$ to $( i/m, (j+1)/n)$, then  $\mathrm{carr}_{I_{m, n}}(y)$ is the edge itself (including the endpoints) and $|\rho_k|$ maps such a $y$ to a point that is again on this edge. Likewise for the horizontal edge from  $( i/m, (j+1)/n)$ to $\left( (i+1)/m, (j+1)/n\right)$.  So far, we have confirmed that for points $y$  in the triangle (including edges and vertices) with vertices as in (\ref{eq: upper left}), we have $|\rho_k|(y) \in \mathrm{carr}_{I_{m, n}}(y)$.  A similar discussion shows the same for points $y$ in the lower-right triangle with vertices $\{   (i/m, j/n), \left( (i+1)/m, (j+1)/n\right), \left( (i+1)/m, j/n\right) \}$.  The result follows.
\end{proof}

Next, we give a version of the usual result on simplicial approximation and homotopy suited to our purposes here. 

\begin{theorem}[Simplicial Approximation and Homotopy] \label{thm: homotopy implies contiguity}
Suppose we have simplicial maps $g, g' \colon (I_{m,n}, \partial I_{m,n}) \to (X, x_0)$ whose spatial realizations are homotopic (relative the boundary); i.e., we have  $|g| \simeq |g'| \colon (I^2, \partial I^2) \to (|X|, x_0)$.  Then for some $k$, we have a contiguity equivalence $g\circ \rho_k \simeq g'\circ \rho_k$. 
\end{theorem}

\begin{proof}
Let $H \colon I^2 \times I \to |X|$ be the (continuous) homotopy (relative the boundary) from $|g|$ to $|g'|$.  Because $I^2 \times I$ is a compact metric space, $H$ is uniformly continuous.  This means that, given any $\epsilon > 0$, there is some $\eta >0$ for which $H\big( B\big( (y, t), \eta \big) \big) \subseteq B\big( H(y, t), \epsilon \big)$ for any point $(y,t) \in I^2 \times I$.  Here, the notation $B( p, r)$ denotes the (open) ball of radius $r$ centered at point $p$ in some metric space.   So, in particular, for any $\epsilon$ there is an $\eta$ such that, if $|t-t'| < \eta$ for $t, t' \in I$, then $H(y, t') \in B\big( H(y, t), \epsilon \big)$ \emph{for all} $y \in I^2$.   

So, let $\epsilon$ be a Lebesgue number for the open covering of $|X|$ that consists of stars of vertices of $X$.  (We know this is an open cover by Proposition~\ref{prop: star cover}.)   Let $\eta$ be such that $|t-t'| < \eta$ implies $H(y, t') \in B\big( H(y, t), \epsilon \big)$ for all $y \in I^2$.  Choose an $n$ such that $1/n < \eta$ and subdivide $I$ into $n$ subintervals of width $1/n$.  Set $t_i = i/n$ for $i= 0, \ldots, n$, and define $f_i(y) = H(y, t_i)$, so that $f_0 = |g|$ and $f_n = |g'|$.      Then for consecutive $f_i, f_{i+1}$ and any $y \in I^2$, we have the pair $\{ f_i(y), f_{i+1}(y) \}$ contained in a ball of radius $\epsilon$ and thus both in $\mathrm{st}(x_j)$ for some vertex $x_j \in X$.  It follows that the sets $f_i^{-1}\big( \mathrm{st}(x_j)\big) \cap f_{i+1}^{-1}\big( \mathrm{st}(x_j)\big)$ form an open cover of $I^2$, for each $i = 0, \ldots, n-1$.    

As in the proof of Theorem~\ref{thm: simp approx}, suppose that $\delta_i$ is a Lebesgue number for the open covering of $I^2$ by the sets $f_i^{-1}\big( \mathrm{st}(x_j)\big) \cap f_{i+1}^{-1}\big( \mathrm{st}(x_j)\big)$. Set 
$$\delta = \mathrm{min} \{ \delta_0, \ldots, \delta_{n-1} \}.$$
 Now let $k$ be large enough such that, for each vertex $a$ of $I_{km, kn}$, the star $\mathrm{st}(a) \subseteq I_{km, kn}$ has diameter less than  $\delta$.  Specifically, we can choose any $k$ for which
 $$2 \sqrt{ \frac{1}{k^2m^2} + \frac{1}{k^2n^2} } < \delta.$$
 Then, for each $i=0, \ldots, n-1$,  define a common simplicial approximation 
 $$g_{i+1} \colon (I_{km, kn}, \partial I_{km, kn}) \to (X, x_0)$$
 to both $f_i$ and $f_{i+1}$ by setting $g_{i+1}(a) = x_0$ for any vertex $a$ of $I_{km,kn}$ on the boundary and setting $g_{i+1}(a) = x_j$ for each non-boundary vertex $a$ of $I_{km,kn}$ with $x_j$ a suitable vertex of $X$ for which $\mathrm{st}(a) \subseteq  f_i^{-1}\big( \mathrm{st}(x_j)\big) \cap f_{i+1}^{-1}\big( \mathrm{st}(x_j)\big)$.  As in Theorem~\ref{thm: simp approx}, we have that $g_{i+1}$ is a simplicial approximation to both $f_i$ and $f_{i+1}$.

We have $g_1$ is a simplicial approximation to $|g|=f_0$ (as well as to $f_1$) and $g_{n}$  is a simplicial approximation to $|g'|=f_n$ (as well as to $f_{n-1}$).  Furthermore, each pair $\{ g_j, g_{j+1} \}$ is a pair of simplicial approximations to the same continuous map $f_{j}$, for $j = 1, \ldots, n-1$.  Thus, by \lemref{lem: same map contig}, we have contiguities $g_j \sim g_{j+1}$  for $j = 1, \ldots, n-1$. 

Finally, observe that $g \circ \rho_k \colon (I_{km, kn}, \partial I_{km, kn}) \to (X, x_0)$ is also a simplicial approximation to $|g|$.  This follows since $g$ is a simplicial approximation to $|g|$ (Remark~\ref{rem: simplicial approximation to self}), $\rho_k$ is a simplicial approximation to $\mathrm{id}$ (Lemma~\ref{lem: rho approx id}), and the composition $g \circ \rho_k$ is thus a simplicial approximation of the composition $|g|\circ \mathrm{id} = |g|$ (see \cite[Cor.22.6]{Bre}, for example).  Likewise, we have $g' \circ \rho_k \colon (I_{km, kn}, \partial I_{km, kn}) \to (X, x_0)$ is a simplicial approximation to $|g'|$. 
Applying \lemref{lem: same map contig} to these last two facts, we now have a sequence of contiguities
$$g \circ \rho_k \sim g_1 \sim \cdots \sim g_n \sim g' \circ \rho_k,$$
whence the continguity equivalence $g \circ \rho_k \simeq  g'\circ \rho_k$.
\end{proof}

\section{Bridging Between $I_{m, n}$ and $I_m \times I_n$}\label{sec: Imn vs. ImxIn}

In the previous section, we showed several results that relate simplicial maps of the form $(I_{m, n}, \partial I_{m,n}) \to (X, x_0)$ and continuous maps of the form $(I^2, \partial I^2) \to (|X|, x_0)$.  In this section, we develop some technical results that relate simplicial maps of the form $\left( I_m \times I_n, \partial(I_m \times I_n) \right) \to (X, x_0)$ and those of the form  $(I_{m, n}, \partial I_{m,n}) \to (X, x_0)$. With these results, we will be able to apply the results of the previous section to show our main result (\thmref{thm: Face group iso homotopy group}).

In our development, we use the simplicial map
$$E\colon I_{m, n} \to I_m \times I_n$$
induced by the vertex map 
$$ \left( \frac{i}{m}, \frac{j}{n} \right) \mapsto (i, j)$$
for each $m, n$ and for each $i, j$ with $0 \leq i \leq m$ and $0 \leq j \leq n$.  Notice that $E^{-1}$, this correspondence of vertices in the other direction, does not give a simplicial map: $\{ (0, 1/m), (1/n, 0) \}$ is not an edge in $I_{m, n}$. 

\begin{lemma}\label{lem: switching E}
With $\alpha_m^{m(k-1)} \colon I_{km} \to I_m$  and $\alpha_n^{n(k-1)} \colon I_{kn} \to I_n$ defined as in \defref{def: extensions}, and $\rho_k \colon I_{km, kn} \to I_{m,n}$ as defined in \defref{def: subdivision}, we have a contiguity equivalence
$$E \circ \rho_k \simeq (\alpha_m^{m(k-1)} \times \alpha_n^{n(k-1)}) \circ E \colon (  I_{km, kn}, \partial  I_{km, kn}) \to \left( I_m \times I_n, \partial(I_m \times I_n) \right).$$
\end{lemma}

\begin{proof}
As in \defref{def: subdivision}, write a typical point in $I_{km, kn}$ in the form
$$\left( \frac{ki+r}{km}, \frac{kj+s}{kn}\right)$$
with $0 \leq i \leq m$, $0 \leq j \leq n$ and $0 \leq r, s < k$. Then we have 
$$E\circ \rho_k\left( (ki+r)/km, (kj+s)/kn \right) = E(i/m, j/n) =(i, j).$$
Now, for $E\colon I_{km, kn} \to I_{km}\times I_{kn}$, we have 
$$E\left( (ki+r)/km, (kj+s)/kn \right) = (ik+r, jk+s).$$
We claim that the composition
$$\alpha_0^{k-1} \circ \alpha_k^{k-1} \circ \cdots \circ \alpha_{(m-1)k}^{k-1} \colon I_{km} \to I_m$$
maps $(ik+r)$ to $i$. 
To see this, note that we have $\alpha_{(i+1)k}^{k-1} \circ \cdots \circ \alpha_{(m-1)k}^{k-1}(ik +r) = ik+r$, since $ik + r< (i+1)k$.
Then, we have $\alpha_{ik}(ik+r) = ik+r-1$,  $\alpha^2_{ik}(ik+r) = ik+r-2$ and so-on, until we have $\alpha^r_{ik}(ik+r) = ik$, after which we have 
$$\alpha^{r+1}_{ik}(ik+r) = \cdots = \alpha^{k-1}_{ik}(ik+r) = ik.$$ 
Finally, we have $\alpha^{k-1}_{(i-1)k}(ik) = (i-1)k+1$,   $\alpha^{k-1}_{(i-2)k}((i-1)k + 1) = (i-2)k+2$, and so-on, until we arrive at 
$$\alpha_0^{k-1} \circ \alpha_k^{k-1} \circ \cdots \circ \alpha_{(i-1)k}^{k-1} (ik) = i.$$
This shows the claim. Likewise, we may show that the composition 
$$\alpha_0^{k-1} \circ \alpha_k^{k-1} \circ \cdots \circ \alpha_{(n-1)k}^{k-1} \colon I_{kn} \to I_n$$
maps $(jk+s)$ to $j$.  It follows that we have agreement of maps
$$E \circ \rho_k = \left( (\alpha_0^{k-1} \circ \cdots \circ \alpha_{(m-1)k}^{k-1})  \times 
( \alpha_0^{k-1} \circ \cdots \circ \alpha_{(m-1)k}^{k-1}) \right)  \circ E \colon  I_{km, kn} \to I_m \times I_n.$$
But we have a contiguity equivalence 
$$\alpha_0^{k-1} \circ \cdots \circ \alpha_{(m-1)k}^{k-1} \simeq  \alpha_m^{k-1} \circ \cdots \circ \alpha_{m}^{k-1} = \alpha_{m}^{m(k-1)} \colon I_{km} \to I_m,$$
from \propref{prop: contiguity results2} and part (a) of \propref{prop: contiguity results1}. Similarly, we have a contiguity equivalence
$$\alpha_0^{k-1} \circ \cdots \circ \alpha_{(n-1)k}^{k-1} \simeq \alpha_{n}^{n(k-1)} \colon I_{kn} \to I_n.$$
Then part (b) of \propref{prop: contiguity results1} gives a contiguity equivalence 
$$\left( (\alpha_0^{k-1} \circ \cdots \circ \alpha_{(m-1)k}^{k-1})  \times 
( \alpha_0^{k-1} \circ \cdots \circ \alpha_{(m-1)k}^{k-1}) \right)  \circ E \simeq (\alpha_m^{m(k-1)} \times \alpha_n^{n(k-1)}) \circ E$$
and the result follows. 
\end{proof}

Via pre-composition with $E$, a simplicial map $I_m \times I_n \to X$ induces one $I_{m, n} \to X$.  But we will want to operate with maps $I_{m, n} \to X$ whether or not they are induced from maps $I_m \times I_n \to X$.  In this regard, the following device is something of a complement to the subdivision maps $\rho_k$ used above.   

\begin{definition}\label{def: gamma2}
For each $m, n$,  define a simplicial map $\gamma \colon I_{2m+1, 2n+1} \to I_{m, n}$ as the simplicial map induced by the vertex map
\[ \gamma\left(  (\frac{2k + \epsilon_1}{2m+1}, \frac{2l+ \epsilon_2}{2n+1}) \right)   = (\frac{k}{m}, \frac{l}{n}),\]
for all $0 \leq k \leq m$, $0 \leq l \leq n$ and $(\epsilon_1, \epsilon_2) \in \{ (0, 0), (1, 0), 0,1), (1,1) \}$. 
\end{definition}

As with the maps $\alpha_i$ from \secref{sec: extn contiguity}, pre-composing a map $g \colon I_{m, n} \to X$ with this $\gamma$ produces a map
$$g\circ \gamma \colon I_{2m+1, 2n+1} \to X,$$
whose values are represented array style by doubling every row and every column of values of $g$ (and re-sizing them so as to fit in the unit square).  Note that $I_{2m+1, 2n+1}$ is not a subdivision of $I_{m, n}$ as the vertices of $I_{m, n}$ are generally not included in those of $I_{2m+1, 2n+1}$.

\begin{lemma}\label{lem: g-gamma and g homotopic}
Let $g\colon \left( I_{m, n}, \partial I_{m, n} \right) \to (X, x_0)$ be  a simplicial map, for any $m, n$.  With $\gamma \colon I_{2m+1, 2n+1} \to I_{m, n}$ as in \defref{def: gamma2}, we have a homotopy of continuous maps
$$|g\circ \gamma| \simeq |g| \colon (I^2, \partial I^2) \to (|X|, x_0)$$
relative the boundary $\partial I^2$.
\end{lemma}

\begin{proof}
First consider the continuous map $|\gamma| \colon I^2 \to I^2$.  Because $I^2$ is a convex subset of the plane, we may use the straight-line homotopy $H \colon I^2 \times I \to I^2$, defined for $(y, t) \in I^2 \times I$ by
$$H(y, t) = (1-t) |\gamma | (y) + ty,$$
to obtain a homotopy from $|\gamma|$ to $\mathrm{id}$.   Note that $H$ preserves the boundary, in the sense that we have $H( \partial I^2, t) \subseteq \partial I^2$ for each $t \in I$. 
This fact follows easily by considering each edge of $\partial I^2$ separately.  For instance, if $y \in I \times \{ 0 \}$ is a point on the bottom edge of the square, we have $|\gamma|(y) \in I \times \{ 0 \}$ and hence $H(y, t) \in I \times \{ 0 \}$.     
Then $|g|\circ H\colon I^2 \times I \to I^2$ gives a homotopy from $|g|\circ |\gamma| = |g \circ \gamma|$ to $|g|$ that is stationary on the boundary $\partial I^2$, since $|g|$ maps the boundary to $x_0$.  
\end{proof}

Note that a  simplicial map $f \colon I_{m, n} \to X$  generally does not map each pair of vertices $\{ (k/m, l/n), ((k+1)/m, (l-1)/n)  \} \in I_{m, n}$ to a simplex in $X$.  This fact means that not
all such maps are induced by a simplicial map $I_{m} \times I_{n} \to X$ with pre-composition by $E \colon I_{m, n} \to I_{m} \times I_{n}$.  
To work around this issue, we give the following construction which will be used in identifying $F(X, x_0)$ with $\pi_2( |X|, x_0)$ in the next section. 

For a simplicial map $f\colon (I_{m, n}, \partial I_{m, n}) \to (X, x_0)$, first use the map $\gamma \colon I_{2m+1, 2n+1} \to I_{m, n}$ of \defref{def: gamma2} to pass to the map
$$f\circ \gamma\colon(I_{2m+1, 2n+1}, \partial I_{2m+1, 2n+1}) \to (X, x_0).$$
Next, define a map $D_f $ on the vertices of $I_{2m+1} \times I_{2n+1}$ as
$$D_f(i, j) = \begin{cases} f\circ \gamma(\dfrac{i}{2m+1}, \dfrac{j-1}{2n+1}) & \text{if $i = 2k+1$ and $j = 2l$}, \\
&\text{ for $1 \leq k \leq m-2$ and $2 \leq l \leq n-1$}\\
f\circ \gamma(\dfrac{i}{2m+1}, \dfrac{j}{2n+1}) & \text{otherwise.}
\end{cases}$$
We remarked above that the vertex map $E^{-1}$ of the simplicial map $E \colon I_{2m+1, 2n+1} \to I_{2m+1}\times I_{2n+1}$ is not a simplicial map.  The definition of $D_f$ here is an adjustment of the vertex map $f\circ \gamma\circ E^{-1}$ at certain isolated (from each other) vertices so that it becomes a simplicial map.  Let us use $\square_{k, l}$ to denote the simplex of  $I_{2m+1}\times I_{2n+1}$
$$\square_{k, l} := \left\{ ( 2k+ \epsilon_1, 2l+ \epsilon_2 ) \mid 0 \leq \epsilon_1, \epsilon_2 \leq 1 \right\},$$
with $0 \leq k \leq m$ and $0 \leq l \leq n$.  Then $D_f$ agrees with the vertex map $f\circ \gamma\circ E^{-1}$ except at  the lower-right corner of each $\square_{k, l}$, for $1 \leq k \leq m-2$ and $2 \leq l \leq n-1$, at which we use instead the value of $f\circ \gamma\circ E^{-1}$ at the upper-right corner of $\square_{k, l-1}$.  Note that we have $f\circ \gamma\circ E^{-1} (\square_{k, l}) = x_0$ for all $\square_{k, l}$ with $k=0, m$ or $l = 0, n$ (the simplices around the boundary of $I_{2m+1}\times I_{2n+1}$) and it is not necessary to adjust the values of $f\circ \gamma\circ E^{-1}$ at the lower-right corners of these squares or of any square whose lower-right corner is next to one of these, to obtain a simplicial map. 
\begin{proposition}\label{prop: Df simplicial}
The vertex map $D_f$ defined above extends to a simplicial map
$$D_f\colon \left(I_{2m+1} \times I_{2n+1}, \partial (I_{2m+1} \times I_{2n+1}) \right) \to (X, x_0).$$
\end{proposition}

\begin{proof}
The only way in which $f\circ \gamma\circ E^{-1}$ fails to preserve simplices is on pairs of vertices $\{ (2k+1, 2l), (2k+2, 2l-1) \}$ for $k$ and $l$ in the ranges on which we will adjust the values to those of $D_f$.  For instance, if $k = 0$.we have  $f\circ \gamma\circ E^{-1}(1, 2l) = f(0, l) = x_0$ and $f\circ \gamma\circ E^{-1}(2, 2l-1) = f(1/m, (l-1)/n)$.  But $x_0 =  f(0, (l-1)/n)$ and $\{ f(0, (l-1)/n), f(1/m, (l-1)/n)\}$ must be a simplex in $X$.  Similarly, for all the values of $k$ and $l$ excluded from the range in which we adjust the values of  $f\circ \gamma\circ E^{-1}$ in the definition of $D_f$, the vertex map $f\circ \gamma\circ E^{-1}$ preserves simplices. 

Now consider the way in which the values of $f\circ \gamma\circ E^{-1}$ change to those of $D_f$ around the particular pair of vertices $\{ (2k+1, 2l), (2k+2, 2l-1) \}$.  These values, before and after, are illustrated in Figures~\ref{fig: f circ rho2  non-digital} and \ref{fig: Df digital}.  In these figures, black segments indicate where the maps are simplicial maps because $f$ is a simplicial map.   Green segments indicate where the maps are simplicial because of the doubling of values.  For instance,  $f\circ \gamma\circ E^{-1}$ maps each simplex $\square_{k, l}$ to a simplex of $X$.  Circled vertices are those on which the value of the map will change from that of $f\circ \gamma\circ E^{-1}$  to that of $D_f$.  The one red-dashed segment indicates the pair of vertices not-necessarily mapped to a simplex of $X$ by $f\circ \gamma\circ E^{-1}$.   In Figure~\ref{fig: Df digital}, the black and green coloring scheme is the same.  Circled vertices are those on which the value of the map has now changed from $f\circ \gamma\circ E^{-1}$ to that of $D_f$.    The various assertions contained in the labeling of these figures are easily checked directly.
\end{proof}

\begin{lemma}\label{lem: D constant}
Let $c_{x_0} \colon I_{m, n} \to X$ be the constant map at $x_0$.  Then $D_{c_{x_0}} \colon I_{2m+1} \times I_{2n+1} \to X$ is the constant map at $x_0$.
\end{lemma}

\begin{proof}
This follows directly from the definitions.
\end{proof}

\begin{figure}[h!]
\[
\begin{tikzpicture}[scale=2]
\foreach \x in {1,...,4} {
 \foreach \y in {1,...,4} {
   \node[inner sep=1.5pt, circle, fill=black]  at (\x,\y) {};
  }
 }
\foreach \x in {1,...,3} {
 \foreach \y in {1,...,4} {
      \draw (\x,\y) -- (\x+1,\y);
 }
}
\foreach \x in {1,...,4} {
 \foreach \y in {1,...,3} {
      \draw (\x,\y) -- (\x,\y+0.8);
 }
}
\foreach \x in {1,...,3} {
 \foreach \y in {1,...,3} {
      \draw (\x,\y) -- (\x+0.8,\y+0.8);
 }
}
\foreach \x in {1,...,3} {
 \foreach \y in {1,3} {
      \draw [color=green, thick](\x+0.2,\y+0.8) -- (\x+1,\y);
 }
}
\foreach \x in {1,3} {
 \foreach \y in {1,2} {
      \draw [color=green, thick](\x+0.2,\y+0.8) -- (\x+1,\y);
 }
}
\draw [color=red, thick, dashed] (2.2,2.8) -- (3,2);
    \node at (0.3,1) {$2l-2$};
    \node at (0.3,2) {$2l-1$};
    \node at (0.3,3) {$2l$};
    \node at (0.3,4) {$2l+1$};
    \node at (1,0.5) {$2k$};
 \foreach \x in {1,...,3} {
    \node at (\x+1,0.5) {$2k+\x$};
  }
\foreach \x in {1,2} {
 \foreach \y in {1,2} {
     \node[below]  at (\x,\y) {$f(\frac{k}{m}, \frac{l-1}{n})$};
  }
}
\foreach \x in {1,2} {
 \foreach \y in {1,2} {
     \node[below]  at (\x,\y+2) {$f(\frac{k}{m}, \frac{l}{n})$};
  }
}
\foreach \x in {1,2} {
 \foreach \y in {1,2} {
     \node[below]  at (\x+2,\y) {$f(\frac{k+1}{m}, \frac{l-1}{n})$};
  }
}
\foreach \x in {1,2} {
 \foreach \y in {1,2} {
     \node[below]  at (\x+2,\y+2) {$f(\frac{k+1}{m}, \frac{l}{n})$};
  }
}
\foreach \x in {2,4} {
 \foreach \y in {1,3} {
    \draw (\x,\y) circle[radius=2pt]  {};
  }
}
\end{tikzpicture}
\]
\caption{Values of the vertex map $f\circ \gamma\circ E^{-1}$  on $[2k, 2k+3] \times [2(l-1), 2l+1] \subseteq I_{2m+1} \times I_{2n+1}$. The pair of vertices $\{ f(\frac{k}{m}, \frac{l}{n}), f(\frac{k+1}{m}, \frac{l-1}{n}) \}$ need not be a simplex in $X$.}
\label{fig: f circ rho2 non-digital}
\end{figure}
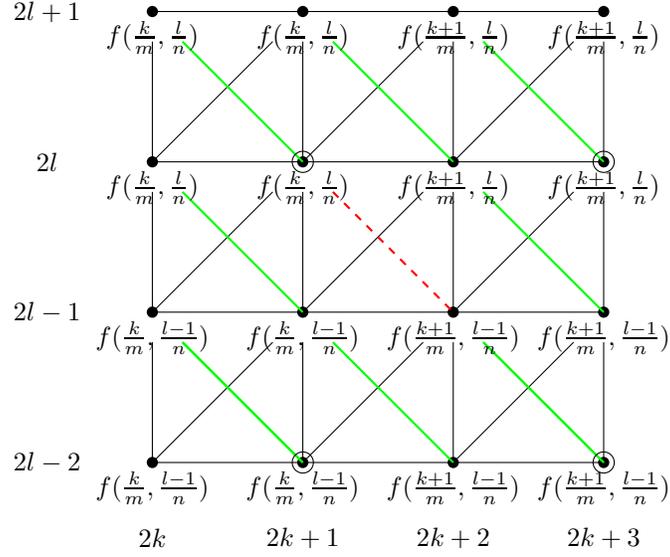 

\begin{figure}[h!]
\[
\begin{tikzpicture}[scale=2]
\foreach \x in {1,...,4} {
 \foreach \y in {1,...,4} {
   \node[inner sep=1.5pt, circle, fill=black]  at (\x,\y) {};
  }
 }
\foreach \x in {1,...,3} {
 \foreach \y in {1,...,4} {
      \draw (\x,\y) -- (\x+1,\y);
 }
}
\foreach \x in {1,...,4} {
 \foreach \y in {1,...,3} {
      \draw (\x,\y) -- (\x,\y+0.8);
 }
}
\foreach \x in {1,...,3} {
 \foreach \y in {1,...,3} {
      \draw (\x,\y) -- (\x+0.8,\y+0.8);
 }
}
\foreach \x in {1,...,3} {
 \foreach \y in {1,2, 3} {
      \draw [color=green, thick](\x+0.2,\y+0.8) -- (\x+1,\y);
 }
}
    \node at (0.3,1) {$2l-2$};
    \node at (0.3,2) {$2l-1$};
    \node at (0.3,3) {$2l$};
    \node at (0.3,4) {$2l+1$};
    \node at (1,0.5) {$2k$};
 \foreach \x in {1,...,3} {
    \node at (\x+1,0.5) {$2k+\x$};
  }
\foreach \x in {1,2} {
 \foreach \y in {1,2} {
     \node[below]  at (\x,\y+\x-1) {$f(\frac{k}{m}, \frac{l-1}{n})$};
  }
}
\foreach \y in {1,2} {
     \node[below]  at (1,\y+2) {$f(\frac{k}{m}, \frac{l}{n})$};
  }
     \node[below]  at (2,4) {$f(\frac{k}{m}, \frac{l}{n})$};
\foreach \x in {1,2} {
 \foreach \y in {1,2} {
     \node[below]  at (\x+2,\y+\x-1) {$f(\frac{k+1}{m}, \frac{l-1}{n})$};
  }
}
\foreach \y in {1,2} {
     \node[below]  at (3,\y+2) {$f(\frac{k+1}{m}, \frac{l}{n})$};
  }
       \node[below]  at (4,4) {$f(\frac{k+1}{m}, \frac{l}{n})$};
\node[below]  at (2,1) {$f(\frac{k}{m}, \frac{l-2}{n})$};
\node[below]  at (4,1) {$f(\frac{k+1}{m}, \frac{l-2}{n})$};
\foreach \x in {2,4} {
 \foreach \y in {1,3} {
    \draw (\x,\y) circle[radius=2pt]  {};
  }
}
\end{tikzpicture}
\]
\caption{Values of $D_f$ on $[2k, 2k+3] \times [2(l-1), 2l+1] \subseteq I_{2m+1} \times I_{2n+1}$. The pair of vertices $\{ f(\frac{k}{m}, \frac{l-1}{n}), f(\frac{k+1}{m}, \frac{l-1}{n}) \}$ is a simplex in $X$.}
\label{fig: Df digital}
\end{figure}
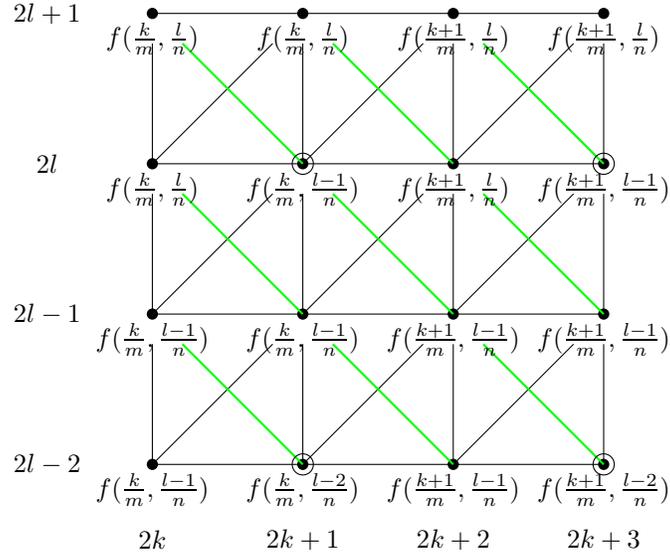 

\begin{proposition}\label{prop: digital f}
Let $f\colon \left( I_{m, n}, \partial I_{m, n} \right)  \to (X, x_0)$ be a simplicial map.
With $D_f$ constructed as above, we have a contiguity 
$$f\circ \gamma  \sim D_f \circ E\colon \left( I_{2m+1, 2n+1}, \partial I_{2m+1, 2n+1}  \right)  \to (X, x_0).$$  
\end{proposition}

\begin{proof}  From the definition of $D_f$,   the maps 
$f\circ \gamma, D_f \circ E\colon I_{2m+1, 2n+1} \to X$ agree at all vertices of $I_{2m+1, 2n+1}$ other than those with coordinates 
$$\left( \dfrac{2k+1}{2m+1}, \dfrac{2l}{2n+1} \right)$$
for $1 \leq k \leq m-2$ and $2 \leq l \leq n-1$.
Since vertices of this form  for different $k$ or $l$ are separated from each other vertically, horizontally and diagonally by intermediate vertices, no two of these occur in a single simplex of $I_{2m+1, 2n+1} $.  Thus, it is sufficient to check the contiguity condition for simplices that contain a typical such vertex to conclude the contiguity $f\circ \gamma  \sim D_f\circ E$.

So, consider a vertex $\left( (2k+1)/(2m+1), (2l)/(2n+1) \right)$ and the $6$ simplices of $I_{2m+1, 2n+1} $ that contain it, as pictured in \figref{fig: hexagon2}.  The $6$ surrounding vertices are labeled with the common values of $f\circ \gamma$ and  $D_f \circ E$.   On the (central) vertex on which the values of of $f\circ \gamma$ and  $D_f \circ E$ differ, we have
$$f\circ \gamma\left( \dfrac{2k+1}{2m+1}, \dfrac{2l}{2n+1} \right) = f\left( \dfrac{k}{m}, \dfrac{l}{n} \right)$$
and
$$ D_f \circ E\left( \dfrac{2k+1}{2m+1}, \dfrac{2l}{2n+1} \right) = f\left( \dfrac{k}{m}, \dfrac{l-1}{n} \right).$$
\begin{figure}[h!] 
\tikzset{vertex/.style={circle,draw,fill,inner sep=0pt,minimum size=1mm}}
\[
\begin{tikzpicture}[scale=2.1,every node/.style={scale=.7}]
\node[vertex,label={}] at (0,0) {};
\node[vertex,label={[right]$f(\frac{k+1}{m},\frac{l}{n})$}] at (1,1) {};
\node[vertex,label={[above]$f(\frac{k}{m},\frac{l}{n})$}] at (0,1) {};
\node[vertex,label={[left]$f(\frac{k}{m},\frac{l}{n})$}] at (-1,0) {};
\node[vertex,label={[left]$f(\frac{k}{m},\frac{l-1}{n})$}] at (-1,-1) {};
\node[vertex,label={[below]$f(\frac{k}{m},\frac{l-1}{n})$}] at (0,-1) {};
\node[vertex,label={[right]$f(\frac{k+1}{m},\frac{l}{n})$}] at (1,0) {};
\draw (1,0) -- (1,1) -- (0,1) -- (-1,0) -- (-1,-1) -- (0,-1) -- (1,0);
\draw (-1,0) -- (1,0);
\draw (0,-1) -- (0,1);
\draw (-1,-1) -- (1,1);
\node at (-0.3,0.3) {$\sigma_1$};
\node at (-0.7,-0.3) {$\sigma_2$};
\node at (-0.3,-0.7) {$\sigma_3$};
\node at (0.3,-0.3) {$\sigma_4$};
\node at (0.7,0.3) {$\sigma_5$};
\node at (0.3,0.7) {$\sigma_6$};
\end{tikzpicture}
\]
\caption{The hexagon of $2$-simplices in $I_{2m+1,2n+1}$ surrounding the vertex $\left( (2k+1)/(2m+1), (2l)/(2n+1) \right)$. Other vertices are labeled with the common values of 
$f\circ \gamma$ and $D_f\circ E$ at those vertices.}\label{fig: hexagon2}
\end{figure}
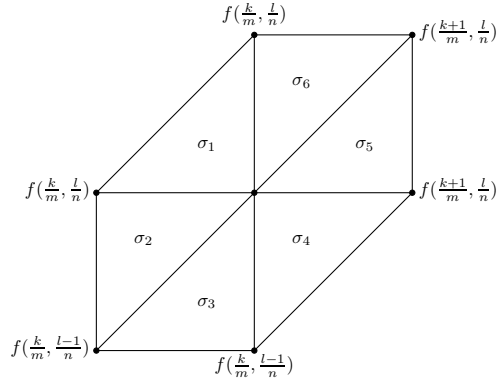
A direct check  shows that we have
$$f\circ \gamma (\sigma_t) \cup D_f\circ E(\sigma_t) =  \begin{cases} f\circ \gamma ( \sigma_2) &\text{ if $t = 1, 2, 3$}, \\
f\circ \gamma ( \sigma_4) &\text{ if $t = 4, 5, 6$}.
\end{cases}
$$
Since $f\circ \gamma$ is a simplicial map, we have $f\circ \gamma (\sigma_t)$ a simplex of $X$ for each $t$, and it follows that we have $f\circ \gamma  \sim D_f\circ E$.  
\end{proof}

With pre-composition by $E$, which we write as $E^*$,  and the construction of $D_f$ from $f$, we have functions between sets of simplicial maps
\bigskip
$$
\left\{\vcenter{ \hbox{\text{Maps}} \hbox{\text{$(I_{m, n}, \partial I_{m, n}) \to (X, x_0)$}} \hbox{\ \ \ \text{all $m$ and $n$} }} \right\} \qquad
\xymatrix{   \ar@<1ex>[r]^-{D_{(-)}} 
&  \ar@<1ex>[l]^-{E^*} } 
\qquad \left\{ \vcenter{ \hbox{\text{Maps}} \hbox{\text{$\left(I_{m} \times I_{n}, \partial I(I_{m} \times I_{n}) \right) \to (X, x_0)$} } \hbox{\text{all $m$ and $n$} }} \right\}
$$
\bigskip

\propref{prop: digital f} describes the effect of the composition $E^* \circ D_{(-)}$.  Namely, up to a contiguity it adjusts $f$ to $f\circ \gamma$, which is a sort of extension contiguity (see \remref{rem: presumed G(X)}---we have not developed this notion in the context of maps $I_{m, n} \to X$).  The next result considers the composition $D_{(-)} \circ E^*$.

\begin{proposition}\label{prop: E then D_f}
Given $g\colon \left(I_{m} \times I_{n}, \partial I(I_{m} \times I_{n}) \right) \to (X, x_0)$.  The maps  
$D_{g\circ E}$ and $g$ are extension-contiguity equivalent: 
we have  $D_{g\circ E} \approx g$. 
\end{proposition}

\begin{proof}
From the definitions, the map $D_{g\circ E}  \left(I_{2m+1} \times I_{2n+1}, \partial I(I_{2m+1} \times I_{2n+1}) \right) \to (X, x_0)$ agrees on vertices with the vertex map $g\circ E \circ \gamma\circ E^{-1}$ except at certain vertices. Thus, 
for $0 \leq \epsilon_1, \epsilon_2 \leq 1$, we have
$$D_{g\circ E}\left( 2k+\epsilon_1, 2l+\epsilon_2 \right) = g\circ E \circ \gamma\left( \dfrac{2k+\epsilon_1}{2m+1}, \dfrac{2l+\epsilon_2}{2n+1} \right) = 
g\circ E\left( \dfrac{k}{m}, \dfrac{l}{n} \right) = g(k, l),$$
except when $(\epsilon_1, \epsilon_2) = (1, 0)$ and  $1 \leq k \leq m-2$ and $2 \leq l \leq n-1$, where we have
$$D_{g\circ E}\left( 2k+1, 2l \right) = g(k, l-1).$$
In the notation of \defref{def: extensions}, let $\alpha_I\colon I_{2m+1} \to I_m$ denote the composition
$$\alpha_I := \alpha_{0} \circ \alpha_{2} \circ \cdots \circ \alpha_{2(m-1)}\circ \alpha_{2m} \colon I_{2m+1} \to I_m$$
and likewise $\alpha_J := \alpha_{0} \circ  \cdots \circ \alpha_{2n} \colon I_{2n+1} \to I_n$. Then we have 
$$(\alpha_I \times \alpha_J)(2k+\epsilon_1, 2l+\epsilon_2) = (k, l)$$
for $0 \leq \epsilon_1, \epsilon_2 \leq 1$ and each $0 \leq k \leq m$, $0 \leq l \leq n$.  This is justified by arguing as we did in the proof of \lemref{lem: switching E} with compositions of the $\alpha_i$.  We leave the details to the reader. Hence, the maps $D_{g\circ E}$ and $g\circ (\alpha_I \times \alpha_J)$ agree at all vertices of $I_{2m+1}  \times I_{2n+1}$ except at some of form $( 2k+1, 2l)$.  We claim there is a contiguity 
$$D_{g\circ E} \sim g\circ (\alpha_I \times \alpha_J)\colon   \left( I_{2m+1}  \times I_{2n+1}, \partial(I_{2m+1}  \times I_{2n+1}) \right) \to (X, x_0).$$
This is argued in a similar way to how we argued in \propref{prop: digital f}.  Vertices of the form $( 2k+1, 2l)$ for different $k$ or $l$ are separated from each other vertically, horizontally and diagonally by intermediate vertices, no two of these occur in a single simplex of $I_{2m+1}  \times I_{2n+1}$.  Thus, it is sufficient to check the contiguity condition for simplices that contain a typical such vertex to conclude the contiguity.
\begin{figure}[h!]
\[
\begin{tikzpicture}[scale=2.5]
\foreach \x in {1, 2,3} {
 \foreach \y in {1,...,3} {
   \node[inner sep=1.5pt, circle, fill=black]  at (\x,\y) {};
  }
 }
\foreach \y in {1,...,3} {
      \draw (1,\y) -- (3,\y);
 } 
\foreach \x in {1,...,3} {
\foreach \y in {1,2} {
      \draw (\x,\y) -- (\x,\y+0.8);
 }
 }
\foreach \x in {1,2} {
\foreach \y in {1,2} {
      \draw (\x,\y) -- (\x+0.8,\y+0.8);
}
}
\foreach \x in {1,2} {
\foreach \y in {1,2} {
      \draw (\x+0.2,\y+0.8) -- (\x+1,\y);
}
}
\draw (1,1) -- (2,2);
\draw (2,1) -- (2,2);
\draw (3,1) -- (2,2);

    \node at (0.5,1) {$2l-1$};
    \node at (0.5,2) {$2l$};
    \node at (0.5,3) {$2l+1$};
    \node at (1,0.5) {$2k$};
    \node at (2,0.5) {$2k+1$};
    \node at (3,0.5) {$2k+2$};


\node[below]  at (1,1) {$g(k, l-1)$};
\node[below]  at (1,2) {$g(k, l)$};
\node[below]  at (1,3) {$g(k, l)$};

\node[below]  at (2,1) {$g(k, l-1)$};
\node[below]  at (2,3) {$g(k, l)$};

\node[below]  at (3,1) {$g(k+1, l-1)$};
\node[below]  at (3,2) {$g(k+1, l)$};
\node[below]  at (3,3) {$g(k+1, l)$};

    \draw (2,2) circle[radius=2pt]  {};

\node at (1.5,2.3) {$s_2$};
\node at (2.5,2.3) {$s_1$};

\node at (1.5,1.3) {$s_3$};
\node at (2.5,1.3) {$s_{4}$};
\end{tikzpicture}
\]
\caption{The $4$ simplices of $I_{2m+1}\times I_{2n+1}$ that contain $(2k+1, 2l)$ as a vertex (circled).  Values of $D_{g\circ E}$ and $g\circ(\alpha_I \times \alpha_J)$ differ only at this vertex. The surrounding vertices are labeled with their common values.}
\label{fig: DgE extn contig g}
\end{figure}
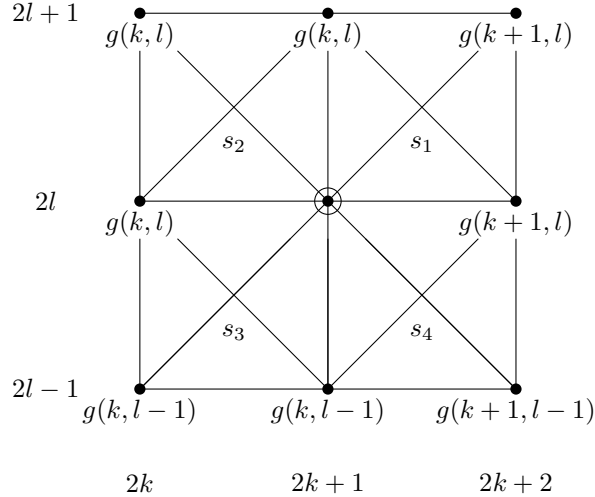 
So, consider a vertex $(2k+1, 2l)$  and the $4$ simplices of $I_{2m+1} \times I_{2n+1} $ that contain it, as pictured in \figref{fig: DgE extn contig g}.  The $8$ surrounding vertices are labeled with the common values of $D_{g\circ E}$ and  $g\circ(\alpha_I \times \alpha_J)$.   On the (circled) vertex on which their values  differ, we have
$$D_{g\circ E}(2k+1, 2l) = g(k, l-1) \qquad \text{and} \qquad g\circ(\alpha_I \times \alpha_J)(2k+1, 2l) = g(k, l).$$
 A direct check shows that we have
$$D_{g\circ E}(s_t) \cup g\circ(\alpha_I \times \alpha_J)(s_t) =  \begin{cases} D_{g\circ E}(\tau) &\text{ if $t = 1$}, \\
D_{g\circ E}(s_3) &\text{ if $t = 2, 3$} \\
D_{g\circ E}(s_4) &\text{ if $t = 4$},
\end{cases}
$$
 where $\tau$ denotes the $2$-simplex $\{ (2k+1, 2l), (2k+2, 2l), (2k+1, 2l-1) \} \subseteq s_1$ of $I_{2m+1} \times I_{2n+1} $.  Since $D_{g\circ E}$ is a simplicial map, the contiguity follows. 
 
From part (b) of \propref{prop: contiguity results2} and part (a) of \propref{prop: contiguity results1}, we deduce a contiguity equivalence 
$$D_{g\circ E} \simeq g\circ (\alpha^{m+1}_m \times \alpha^{n+1}_n)\colon   \left( I_{2m+1}  \times I_{2n+1}, \partial(I_{2m+1}  \times I_{2n+1}) \right) \to (X, x_0),$$
with $g\circ (\alpha^{m+1}_m \times \alpha^{n+1}_n)$ a trivial extension of $g$.  The result follows.
\end{proof}

\begin{lemma}\label{lem: extension of contiguity} 
Suppose we have a contiguity of simplicial maps 
$$g \sim g' \colon(I_{m, n}, \partial I_{m, n}) \to (X, x_0).$$
Let $D_g, D_{g'} \colon \left(I_{2m+1} \times I_{2n+1}, \partial (I_{2m+1} \times I_{2n+1}) \right) \to (X, x_0)$ be the simplicial maps constructed from $g$ and $g'$ of \propref{prop: Df simplicial}.  Then we have a contiguity equivalence  
$$D_g \simeq D_{g'} \colon \left(I_{2m+1} \times I_{2n+1}, \partial I(I_{2m+1} \times I_{2n+1}) \right) \to (X, x_0).$$
\end{lemma}

\begin{proof}
First consider the case in which the contiguity  $g \sim g'$ is a ``spider move\footnote{Terminology adopted from Laura Scull.}" that changes the value of $g$ at exactly one vertex of $I_{m, n}$.  We will show that in this special case, we actually have a contiguity $D_g \sim D_{g'} \colon \left(I_{2m+1} \times I_{2n+1}, \partial (I_{2m+1} \times I_{2n+1}) \right) \to (X, x_0)$. So, suppose we have $g \sim g'$ and for some particular  $(k/m, l/n) \in I_{m, n}$, with $1 \leq k \leq m-1$ and $1 \leq l \leq n-1$ we have 
$$g'(i/m, j/n) =  g(i/m, j/n)$$
for all $(i, j) \not= (k, l)$.  The values $g(k/m, l/n)$ and $g'(k/m, l/n)$, together with the common values of $g(i/m, j/n)=g'(i/m, j/n)$ on adjacent vertices of $I_{m, n}$, must satisfy the contiguity condition.  These data are summarized in Figure~\ref{fig: hexagong}.  In the figure we have circled the (only) vertex at which the values of $g$ and $g'$ differ.  The contiguity condition implies that $g(\sigma_t) \cup g'(\sigma_t)$ span a simplex in $X$, for $t=1, \ldots, 6$.   For example, evaluating on $\sigma_1$ we have that
$$g(\sigma_1) \cup g'(\sigma_1) = \{ g((k-1)/m, l/n), g(k/m, (l+1)/n), g(k/m, l/n), g'(k/m, l/n) \}$$
is a simplex of $X$.

\begin{figure}[h!] 
\tikzset{vertex/.style={circle,draw,fill,inner sep=0pt,minimum size=1mm}}
\[
\begin{tikzpicture}[scale=2.1,every node/.style={scale=.7}]
\node[vertex,label={}] at (0,0) {};
\node[vertex,label={[right]$g(\frac{k+1}{m},\frac{l+1}{n})$}] at (1,1) {};
\node[vertex,label={[above]$g(\frac{k}{m},\frac{l+1}{n})$}] at (0,1) {};
\node[vertex,label={[left]$g(\frac{k-1}{m},\frac{l}{n})$}] at (-1,0) {};
\node[vertex,label={[left]$g(\frac{k-1}{m},\frac{l-1}{n})$}] at (-1,-1) {};
\node[vertex,label={[below]$g(\frac{k}{m},\frac{l-1}{n})$}] at (0,-1) {};
\node[vertex,label={[right]$g(\frac{k+1}{m},\frac{l}{n})$}] at (1,0) {};
\draw (1,0) -- (1,1) -- (0,1) -- (-1,0) -- (-1,-1) -- (0,-1) -- (1,0);
\draw (-1,0) -- (1,0);
\draw (0,-1) -- (0,0);
\draw (0,.2) -- (0,1);

\draw (-1,-1) -- (0,0);
\draw (.25,.25) -- (1,1);

\node at (-0.3,0.3) {$\sigma_1$};
\node at (-0.7,-0.3) {$\sigma_2$};
\node at (-0.3,-0.7) {$\sigma_3$};
\node at (0.3,-0.3) {$\sigma_4$};
\node at (0.7,0.3) {$\sigma_5$};
\node at (0.3,0.7) {$\sigma_6$};
\draw (0,0) circle[radius=2pt]  {};
\node[label={[above] $g(\frac{k}{m}, \frac{l}{n}) \text{ or } g'(\frac{k}{m}, \frac{l}{n}) $}] at (0,0) {};
\end{tikzpicture}
\]
\caption{The hexagon of $2$-simplices in $I_{m,n}$ that contain $(k/m, l/n)$ as a vertex. Surrounding vertices are labeled with the (common) values $g$ and $g'$ take at them.
Values of $g$ and $g'$ differ only at $(k/m, l/n)$.}\label{fig: hexagong}
\end{figure}
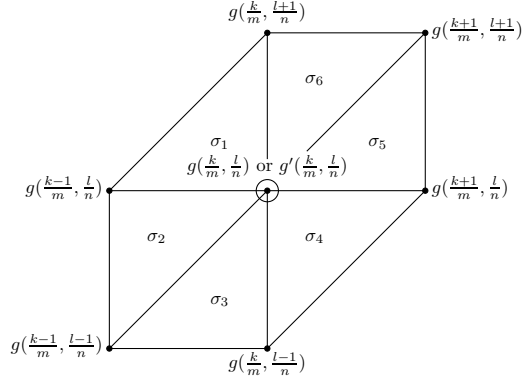

\begin{figure}[h!]
\[
\begin{tikzpicture}[scale=2]
\foreach \x in {2,3} {
 \foreach \y in {1,...,5} {
   \node[inner sep=1.5pt, circle, fill=black]  at (\x,\y) {};
  }
 }
 \foreach \y in {1,...,4} {
   \node[inner sep=1.5pt, circle, fill=black]  at (1,\y) {};
 }
  \foreach \y in {2,...,5} {
   \node[inner sep=1.5pt, circle, fill=black]  at (4,\y) {};
 }
\foreach \y in {1,...,4} {
      \draw (1,\y) -- (2,\y);
 } 
 \foreach \y in {1,...,5} {
      \draw (2,\y) -- (3,\y);
 }
 \foreach \y in {2,...,5} {
      \draw (3,\y) -- (4,\y);
 } 
\foreach \y in {1,...,3} {
      \draw (1,\y) -- (1,\y+0.8);
 }
\foreach \x in {2,3} {
 \foreach \y in {1,...,4} {
      \draw (\x,\y) -- (\x,\y+0.8);
 }
}
\foreach \y in {2,...,4} {
      \draw (4,\y) -- (4,\y+0.8);
 }
\foreach \y in {1,...,3} {
      \draw (1,\y) -- (1+0.8,\y+0.8);
}
\foreach \y in {1,...,4} {
      \draw (2,\y) -- (2+0.8,\y+0.8);
}
\foreach \y in {2,...,4} {
      \draw (3,\y) -- (3+0.8,\y+0.8);
}
\foreach \y in {1,2, 3} {
      \draw (1+0.2,\y+0.8) -- (2,\y);
}
\foreach \y in {1,2, 3, 4} {
      \draw (2+0.2,\y+0.8) -- (3,\y);
}
\foreach \y in {2, 3, 4} {
      \draw (3+0.2,\y+0.8) -- (4,\y);
}

    \node at (0.3,1) {$2l-1$};
    \node at (0.3,2) {$2l$};
     \foreach \y in {1,...,3} {
    \node at (0.3,\y+2) {$2l+\y$};
  }
    \node at (1,0.5) {$2k-1$};
    \node at (2,0.5) {$2k$};
 \foreach \x in {1,2} {
    \node at (\x+2,0.5) {$2k+\x$};
  }
\foreach \y in {1,2} {
     \node[below]  at (1,\y) {$g(\frac{k-1}{m}, \frac{l-1}{n})$};
}
\foreach \y in {3,4} {
     \node[below]  at (1,\y) {$g(\frac{k-1}{m}, \frac{l}{n})$};
}

\node[below]  at (2,1) {$g(\frac{k}{m}, \frac{l-1}{n})$};
\foreach \y in {4,5} {
     \node[below]  at (2,\y) {$g(\frac{k}{m}, \frac{l+1}{n})$};
}
\foreach \y in {1,2} {
     \node[below]  at (3,\y) {$g(\frac{k}{m}, \frac{l-1}{n})$};
}
\node[below]  at (3,5) {$g(\frac{k}{m}, \frac{l+1}{n})$};
\foreach \y in {2,3} {
     \node[below]  at (4,\y) {$g(\frac{k+1}{m}, \frac{l}{n})$};
}
\foreach \y in {4,5} {
     \node[below]  at (4,\y) {$g(\frac{k+1}{m}, \frac{l+1}{n})$};
}
\foreach \x in {2,3} {
    \draw (\x,\x+1) circle[radius=2pt]  {};
    \draw (\x,\x) circle[radius=2pt]  {};  
}
\draw (1,1) -- (2,2);
\draw (2,1) -- (2,2);
\draw (3,1) -- (2,2);

\draw (1,2) -- (2,3);
\draw (2,2) -- (2,3);
\draw (3,2) -- (2,3);

\draw (2,2) -- (3,3);
\draw (3,2) -- (3,3);
\draw (4,2) -- (3,3);

\draw (2,3) -- (3,4);
\draw (3,3) -- (3,4);
\draw (4,3) -- (3,4);

\node at (2.5,4.3) {$s_1$};
\node at (3.5,4.3) {$s_2$};
\node at (1.5,3.3) {$s_3$};
\node at (2.5,3.3) {$s_4$};
\node at (3.5,3.3) {$s_5$};
\node at (1.5,2.3) {$s_6$};
\node at (2.5,2.3) {$s_7$};
\node at (3.5,2.3) {$s_8$};
\node at (1.5,1.3) {$s_9$};
\node at (2.5,1.3) {$s_{10}$};
\end{tikzpicture}
\]
\caption{Values of $D_g$ and $D_{g'}$ on the four vertices on which they differ and the $10$ simplices of $I_{2m+1}\times I_{2n+1}$ that involve them.}
\label{fig: Dg contig Dg'}
\end{figure}
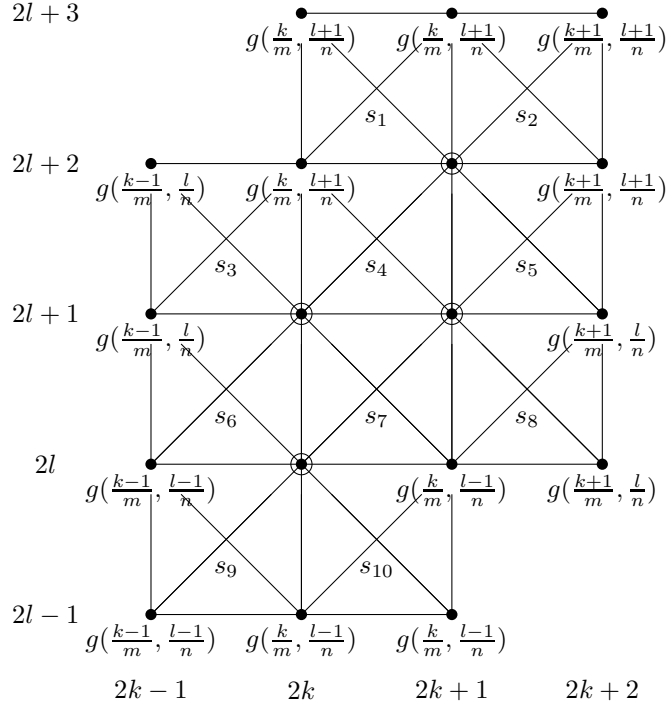 

Now consider the maps   $D_g, D_{g'} \colon \left(I_{2m+1} \times I_{2n+1}, \partial (I_{2m+1} \times I_{2n+1}) \right) \to (X, x_0)$.  Following the description of these maps from above, we see that they differ in value only at the four vertices $\{ (2k, 2l), (2k, 2l+1), (2k+1, 2l+1), (2k+1, 2l+2) \}$ of $I_{2m+1} \times I_{2n+1}$ (recall that, for $D_g$ for instance, we changed the values from those of the vertex map $g\circ \gamma \circ E^{-1}$ at the bottom-right corner of each block of four vertices resulting from the two-fold subdivision).   These vertices and the simplices that surround them are pictured in Figure~\ref{fig: Dg contig Dg'}.  The $4$ vertices on which the values of $D_g$ and $D_{g'}$ differ are circled.  On these four vertices, $D_g$ takes the value $g(k/m, l/n)$ and $D_{g'}$ takes the value $g'(k/m, l/n)$. To check the contiguity $D_g \sim D_{g'}$ we just need to check that for each of the $10$ marked simplices, we have that
 $D_g(s_t) \cup D_{g'}(s_t)$ is a simplex of $X$.    We do this directly, referring to the data from \figref{fig: hexagong} for each simplex $\sigma_1, \ldots, \sigma_6$ and \figref{fig: Dg contig Dg'} for each simplex $s_1, \ldots, s_{10}$.  We find as follows:
$$D_g(s_t) \cup D_{g'}(s_t)  = \{ g(k/m, (l+1)/n), g(k/m, l/n), g'(k/m, l/n) \} \subseteq g(\sigma_1) \cup g'(\sigma_1)$$
for $t = 1, 4$ and 
$$D_g(s_t) \cup D_{g'}(s_t)  = \{ g(k/m, (l-1)/n), g(k/m, l/n), g'(k/m, l/n) \} \subseteq g(\sigma_3) \cup g'(\sigma_3)$$
for $t = 7, 10$.  On the other simplices, we leave the reader to check that we have   
$$\begin{aligned}
D_g(s_2) \cup D_{g'}(s_2) & =  g(\sigma_6) \cup g'(\sigma_6)\\ 
D_g(s_3) \cup D_{g'}(s_3) & =  g(\sigma_1) \cup g'(\sigma_1)\\ 
D_g(s_5) \cup D_{g'}(s_5) & =  g(\sigma_5) \cup g'(\sigma_5)\\ 
D_g(s_6) \cup D_{g'}(s_6) & =  g(\sigma_2) \cup g'(\sigma_2)\\ 
D_g(s_8) \cup D_{g'}(s_8) & =  g(\sigma_4) \cup g'(\sigma_4)\\ 
D_g(s_9) \cup D_{g'}(s_9) & =  g(\sigma_3) \cup g'(\sigma_3) 
\end{aligned}$$
In all cases, the contiguity $g \sim g'$ implies that the union $D_g(s_t) \cup D_{g'}(s_t)$ is (at least) contained in a simplex of $X$, and since a subset of a simplex is again a simplex, is in fact equal to a simplex of $X$.  Since $D_g$ and $D_{g'}$ agree on all other simplices of $I_{2m+1} \times I_{2n+1}$, the 
contiguity $D_g \sim D_{g'}$ follows.  

Now a general contiguity $g \sim g' \colon(I_{m, n}, \partial I_{m, n}) \to (X, x_0)$ may be broken down into a contiguity equivalence 
$$g \sim \cdots \sim g_i \sim \cdots \sim g'$$
consisting of a succession of spider moves.  This is not difficult to see; we leave a justification to the reader.
From the first part, this will give a succession of contiguities
$$D_g \sim \cdots \sim D_{g_i} \sim \cdots \sim D_{g'},$$
namely,  a contiguity equivalence $D_g \simeq D_{g'}\colon \left(I_{2m+1} \times I_{2n+1}, \partial (I_{2m+1} \times I_{2n+1}) \right) \to (X, x_0)$.  
\end{proof}

\begin{remark}\label{rem: presumed G(X)}
We believe it is possible to mimic the (trivial) extensions of \secref{sec: extn contiguity} in the context of simplicial maps $(I_{m, n}, \partial I_{m,n}) \to (X, x_0)$, and define a group in that context in a way very much like how we have defined our face group using simplicial maps $\left(I_{2m+1} \times I_{2n+1}, \partial (I_{2m+1} \times I_{2n+1}) \right) \to (X, x_0)$.  Some technical complications arise with inverses and also with having to re-size the domain after extending.  Assuming such a development, leading to a group $G(X)$, say, then it should likewise be possible to develop the maps $E^*$ and $D_{(-)}$ from above \propref{prop: E then D_f}, together with more results along the lines of those in this section, into an isomorphism $G(X) \cong F(X, x_0)$.  Finally, the results of  \secref{sec: simp approx} could presumably be developed further to obtain a second isomorphism  $G(X) \cong \pi_2( |X|, x_0)$.  Implementing all this would take considerable work beyond what we have presented here.  Instead, 
we have developed results involving maps  $(I_{m, n}, \partial I_{m,n}) \to (X, x_0)$ only to the extent that we need them to establish the isomorphism of \thmref{thm: Face group iso homotopy group} directly, without involving the presumed intermediate group $G(X)$.         
\end{remark}

\section{Identification of $F(X, x_0)$ with $\pi_2(|X|, x_0)$}\label{sec: face group iso}  

We choose the formulation of  $\pi_2(|X|, x_0)$, the second homotopy group of the spatial realization of a simplicial complex $X$, that most closely fits our 
purposes here.  Namely, $\pi_2(|X|, x_0)$ consists of equivalence classes of maps of pairs $(I^2, \partial I^2) \to (X, x_0)$ modulo the equivalence relation of homotopy relative the boundary (a homotopy must remain stationary at $x_0$ on the boundary $\partial I^2$).  The product in this group is that induced by the operation on maps  
$$f \cdot g (i, j) = 
\begin{cases} f(2i, j) & 0 \leq i \leq 1/2\\
g(2i-1, j) & 1/2 \leq i \leq 1,
\end{cases}
$$
with $f, g\colon (I^2, \partial I^2) \to (X, x_0)$.

\begin{theorem}\label{thm: Face group iso homotopy group}
Let $(X, x_0)$ be a based simplicial complex.  Define a map 
$$\Phi\colon F(X, x_0) \to \pi_2( |X|, x_0)$$
by $\Phi\big( [g] \big) = [ |g\circ E| ]$ for a simplicial map $g \colon \left(I_{m} \times I_{n}, \partial (I_{m} \times I_{n}) \right) \to (X, x_0)$ and $E\colon I_{m, n} \to I_m \times I_n$ the simplicial map from \secref{sec: Imn vs. ImxIn}.  Then $\Phi$ is an isomorphism of groups. 
\end{theorem}

\begin{proof}
We begin by showing that $\Phi$ is well defined.  First, consider 
$$\overline{g} = g\circ (\alpha_m^p \times \alpha_n^q) \colon \left(I_{m+p} \times I_{n+q}, \partial (I_{m+p} \times I_{n+q}) \right) \to (X, x_0)$$
a trivial extension of $g$, with $p, q \geq 0$. Visually, we may represent $|\overline{g}\circ E| \colon  (I^2, \partial I^2) \to (|X|, x_0)$ as in the following figure. 
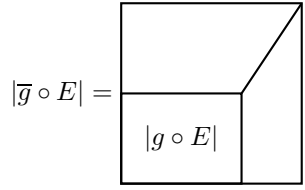
\[
|\overline{g}\circ E|  =
\vcenter{\hbox{\begin{tikzpicture}[scale=.4]
\draw[thick] (0,0) rectangle (4,3) node[pos=.5] {$|g\circ E|$};
\draw[thick] (4, 3)--(6, 6);
\draw[thick] (0,0) rectangle (6,6);
\end{tikzpicture}
}}
\]
As in our previous such figures in \secref{sec: defn of Face Group}, the areas left blank are mapped to $x_0$. Unlike those previous figures, which indicated vertex maps, this figure indicates a continuous map: the rectangle $[0, m/(m+p)] \times [0, n/(n+q)]$ is an actual rectangle in $\R^2$; the whole of the blank regions and their edges are mapped to $x_0$.
We claim that there is a homotopy relative the boundary $|\overline{g}\circ E| \simeq |g\circ E| \colon  (I^2, \partial I^2) \to (|X|, x_0)$. For this, we want a map $\delta \colon I^2 \to I^2$ that expands the rectangle to the whole square $I^2$, compresses the trapeziods to their top and right edges, and satisfies $\delta(\partial I^2) \subseteq \partial I^2$.  The following will suffice.   

A typical point in the rectangle is $(rm/(m+p), sn/(n+q)) \in  [0, m/(m+p)] \times [0, n/(n+q)]$, for $0 \leq r, s \leq 1$.  We set
$$\delta\left( \frac{rm}{m+p},  \frac{sn}{n+q}\right) := (r, s) \in I^2.$$
We extend $\delta$ to the whole of $I^2$ as follows. The typical point on the top edge of the rectangle is $y_r:=(rm/(m+p), n/(n+q))$ for some $0 \leq r \leq 1$ and $\delta$ as we have defined it so far maps this point to $(r, 1)$ on the top edge of the square.  So, on  the trapezoid above the rectangle, set
$$\delta\left( \overline{ y_r\  (r, 1)}  \right) := (r, 1)$$
for all points on the segment with endpoints $(rm/(m+p), n/(n+q))$ and $(r, 1)$ for each $0 \leq r \leq 1$.
Similarly, on  the trapezoid to the right of the rectangle, set
$$\delta\left( \overline{ ( \frac{m}{m+p},  \frac{sn}{n+q} ) (1, s)}  \right) := (1, s)$$
for all points on the segment with endpoints $(m/(m+p), sn/(n+q))$ and $(1, s)$ for each $0 \leq s \leq 1$.  This defines our map $\delta \colon I^2 \to I^2$.  Since $I^2$ is a convex subset of $\R^2$, the straight-line homotopy $H \colon I^2 \times I \to I^2$ defined by
$$H( y, t) = (1-t) \delta(y) + t y$$
for points $y \in I^2$ gives a homotopy $\delta \simeq \mathrm{id} \colon I^2 \to I^2$.   Because $\delta(\partial I^2) \subseteq \partial I^2$ and $|g \circ E| (\partial I^2) = \{ x_0 \}$, the composition 
$|g \circ E| \circ H \colon (I^2, \partial I^2) \to (|X|, x_0)$ gives a homotopy $|g \circ E| \circ \delta \simeq |g \circ E|$.  But $|g \circ E| \circ \delta = |\overline{g}\circ E|$ and our claim follows.
  
Second, suppose that we have contiguous maps 
$$g \sim g'  \colon \left(I_{m} \times I_{n}, \partial (I_{m} \times I_{n}) \right) \to (X, x_0).$$
Then we have a contiguity 
$$g\circ E \sim g'\circ E  \colon (I_{m, n}, \partial I_{m, n} )  \to (X, x_0).$$
Since contiguous maps have homotopic spatial realizations via the straight line homotopy in each simplex, it follows that we have a homotopy relative the boundary $|g\circ E| \simeq |g' \circ E| \colon \colon (I^2, \partial I^2) \to (|X|, x_0)$.  From the last two points, it follows that $\Phi$ is well defined.  

For $\Phi$ to be a homomorphism, we require a homotopy relative the boundary
 $$| (g_1 \cdot g_2)\circ E| \simeq |g_1\circ E| \cdot |g_2\circ E| \colon (I^2, \partial I^2) \to (|X|, x_0)$$
 for elements $[g_1], [g_2] \in F(X, x_0)$.  As $\Phi$ is well-defined, we may use trivial extensions if need be to choose representatives 
$$g_1, g_2\colon \left(I_{m} \times I_{n}, \partial (I_{m} \times I_{n}) \right) \to (X, x_0)$$
with the same-sized domain.
The operation $g_1 \cdot g_2$ is our operation on simplicial maps from \secref{sec: defn of Face Group}, and the operation $|g_1\circ E| \cdot |g_2\circ E|$ is addition of continuous maps $(I^2, \partial I^2) \to (|X|, x_0)$ as given at the start of this section. 
 Pictorially, then, we may represent these two maps as
\[
| (g_1 \cdot g_2)\circ E| =
\vcenter{\hbox{\begin{tikzpicture}[scale=.4]
\draw[thick] (0,0) rectangle (4,4) node[pos=.5] {$|g_1\circ E|$};
\draw[thick] (0,4) rectangle (4,8);
\draw[thick] (4,4) rectangle (8,8) node[pos=.5] {$|g_2\circ E|$};
\draw[thick] (4,0) rectangle (8,4);
\end{tikzpicture}
}}
\quad \text{and} \quad 
|g_1\circ E| \cdot |g_2\circ E|  =
\vcenter{\hbox{\begin{tikzpicture}[scale=.4]
\draw[thick] (0,0) rectangle (4,8) node[pos=.5] {$|g_1\circ E|$};
\draw[thick] (4,0) rectangle (8,8) node[pos=.5] {$|g_2\circ E|$};
\end{tikzpicture}
}}
\]
A homotopy from one to the other may be represented pictorially in two steps (in which $\simeq$ denotes homotopy relative the boundary of maps $I^2 \to |X|$)
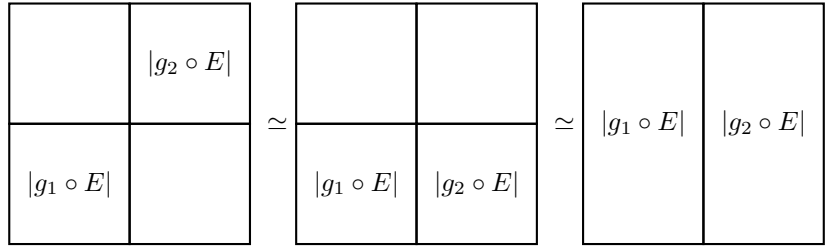
$$
\vcenter{\hbox{\begin{tikzpicture}[scale=.4]
\draw[thick] (0,0) rectangle (4,4) node[pos=.5] {$|g_1\circ E|$};
\draw[thick] (0,4) rectangle (4,8);
\draw[thick] (4,4) rectangle (8,8) node[pos=.5] {$|g_2\circ E|$};
\draw[thick] (4,0) rectangle (8,4);
\end{tikzpicture}
}}
\simeq
\vcenter{\hbox{\begin{tikzpicture}[scale=.4]
\draw[thick] (0,0) rectangle (4,4) node[pos=.5] {$|g_1\circ E|$};
\draw[thick] (0,4) rectangle (4,8);
\draw[thick] (4,4) rectangle (8,8);
\draw[thick] (4,0) rectangle (8,4)  node[pos=.5] {$|g_2\circ E|$};
\end{tikzpicture}
}}
\simeq
\vcenter{\hbox{\begin{tikzpicture}[scale=.4]
\draw[thick] (0,0) rectangle (4,8) node[pos=.5] {$|g_1\circ E|$};
\draw[thick] (4,0) rectangle (8,8) node[pos=.5] {$|g_2\circ E|$};
\end{tikzpicture}
}}
$$
Symbolically, the central figure represents  an intermediate  map $F \colon (I^2, \partial I^2) \to (X, x_0)$ given by
$$F(i, j) = \begin{cases}
| g_1\circ E|(2i, 2j) & (i, j) \in [0, 1/2]\times[0, 1/2] \\
| g_2\circ E|(2i-1, 2j) & (i, j) \in [1/2, 1]\times[0, 1/2] \\
x_0 & (i, j) \in I\times[1/2, 1] 
\end{cases}$$ 
The first homotopy $| (g_1 \cdot g_2)\circ E| \simeq F \colon (I^2, \partial I^2) \to (|X|, x_0)$ may be written as follows. Assemble a homotopy $H \colon I^2 \times I \to I$ from
$$H_1 \colon ([0, 1/2] \times I)\times I \to I^2 \qquad \text{and} \qquad H_2 \colon ([1/2, 1] \times I)\times I \to I^2$$
where $H_1$ is stationary with formula 
$$H_1( (i, j), t) = \begin{cases} 
(2i, 2j) &  (i, j) \in [0, 1/2]\times[0, 1/2] \\
(2i, 1) &(i, j) \in [0, 1/2]\times[1/2, 1]
\end{cases}$$
and $H_2$ is given by
$$H_2( (i, j), t) = \begin{cases} 
(2i, 0) &  (i, j) \in [1/2, 1]\times[0, t/2] \\
(2i, 2(j - t/2) ) &(i, j) \in [1/2, 1]\times[t/2, (t+1)/2]\\
(2i, 1) &(i, j) \in [1/2, 1]\times[(t+1)/2, 1].
\end{cases}$$
Then define $G_1 \colon I^2 \times I \to X$ by
$$G_1( y, t)  = \begin{cases}
|g_1 \circ E|\circ H_1(y, t) & y \in [0, 1/2]\times I\\
|g_2 \circ E|\circ H_2(y, t) & y \in [1/2, 1]\times I. 
\end{cases}$$
We leave the reader to check that this gives a homotopy relative the boundary $F \simeq | (g_1 \cdot g_2)\circ E| \colon (I^2, \partial I^2) \to (|X|, x_0)$.  Starting this homotopy at $F$, as we have done, eases the formulas above.  Since homotopy is a symmetric relation, we obtain one  $| (g_1 \cdot g_2)\circ E| \simeq F$.  
The second homotopy $G_2 \colon (I^2, \partial I^2) \to (|X|, x_0)$ may be written directly as 
$$G_2 ((i, j), t) = \begin{cases}
|g_1\circ E|(2i, (2-t)j) & (i, j) \in [0, 1/2]\times[0, 1/(2-t)] \\
|g_2\circ E|(2i-1, (2-t)j) & (i, j) \in [1/2, 1]\times[0, 1/(2-t)] \\
x_0 & (i, j) \in I\times[1/(2-t), 1].
\end{cases}$$
It is straightforward to check that this gives a homotopy relative the boundary
$F \simeq  |g_1\circ E| \cdot |g_2\circ E| \colon (I^2, \partial I^2) \to (|X|, x_0)$.  It now follows that $\Phi$ is a homomorphism.

It remains to show that $\Phi$ is surjective and injective.  This last part of the argument is the (only) reason we developed the results of \secref{sec: simp approx} and  \secref{sec: Imn vs. ImxIn}.  Suppose we have an element $[f] \in \pi_2(|X|, x_0)$ represented by a continuous map $f \colon (I^2, \partial I^2) \to (|X|, x_0)$.  Theorem~\ref{thm: simp approx} gives a simplical approximation $g \colon (I_{m, m}, \partial I_{m, m}) \to (X, x_0)$, for some $m$, that satisfies $[f] = [|g|] \in \pi_2(|X|, x_0)$.  Then we have an element $[D_g] \in F(X, x_0)$ represented by the map $D_g \colon \left( I_{2m} \times I_{2n}, \partial(I_{2m}\times I_{2n})\right) \to (X, x_0)$, as in \propref{prop: Df simplicial}.  Furthermore, we have a contiguity 
$$D_g \circ E \sim g\circ\gamma \colon (I_{2m+1, 2m+1}, \partial I_{2m+1, 2m+1}) \to (X, x_0)$$
by \propref{prop: digital f}.   Now contiguous maps have homotopic spatial realizations so, together with \lemref{lem: g-gamma and g homotopic}, we obtain homotopies relative the boundary
$$|D_g \circ E| \simeq |g\circ \gamma| \simeq |g| \colon    (I^2, \partial I^2) \to (|X|, x_0).$$
Notice that the first homotopy that comes from the contiguity is indeed relative the boundary since that homotopy is the straight-line homotopy in (the spatial realization of) each simplex, and so the homotopy will be stationary at the constant map on the boundary. Finally, we may write 
$$\Phi( [D_g]) = [|D_g\circ E|] = [|g\circ \gamma|] = [|g|] = [f],$$
and hence $\Phi$ is surjective.  

Since we have shown $\Phi$ is a homomorphism, it is sufficient for injectivity to show that $[g] = 0 \in F(X, x_0)$ whenever we have $\Phi( [g]) = 0 \in \pi_2(X, x_0)$. So, assume that some simplicial map $g \colon \left( I_{m} \times I_{n}, \partial(I_{m}\times I_{n})\right) \to (X, x_0)$ has $|g\circ E|$ homotopic relative the boundary to the constant map at $x_0$.  Then for the maps 
$g, c_{x_0} \colon \left( I_{m} \times I_{n}, \partial(I_{m}\times I_{n})\right) \to (X, x_0)$ we have a homotopy relative the boundary  
$|g\circ E| \simeq |c_{x_0} \circ E| \colon (I^2, \partial I^2) \to (|X|, x_0)$. From Theorem~\ref{thm: homotopy implies contiguity}, we have a contiguity equivalence
$$g\circ E \circ \rho_k \simeq c_{x_0}\circ E \circ \rho_k\colon ( I_{km, kn}, \partial I_{km, kn} ) \to (X, x_0)$$
for some $k$.  Notice that  $c_{x_0}\circ E \circ \rho_k$ is again a constant map at $x_0$ which, by a slight abuse of notation, we also write as $c_{x_0}$.  Also, \lemref{lem: switching E} 
gives a contiguity equivalence $E \circ \rho_k \simeq (\alpha_m^{m(k-1)} \times \alpha_n^{n(k-1)}) \circ E$.  Part (a) of \lemref{prop: contiguity results1} gives a contiguity equivalence
$$\overline{g} \circ E \simeq c_{x_0} \colon ( I_{km, kn}, \partial I_{km, kn} ) \to (X, x_0),$$
where $\overline{g} = g\circ (\alpha_m^{m(k-1)} \times \alpha_n^{n(k-1)}) \colon \left( I_{km} \times I_{kn}, \partial( I_{km} \times I_{kn})   \right) \to (X, x_0)$, a trivial extension of $g$. 
 Now \lemref{lem: extension of contiguity} gives a contiguity equivalence 
 $$D_{\overline{g} \circ E} \simeq D_{c_{x_0}} \colon  \left( I_{2km+1} \times I_{2kn+1}, \partial( I_{2km+1} \times I_{2kn+1})   \right) \to (X, x_0).$$
 \lemref{lem: D constant} gives $D_{c_{x_0}}  = c_{x_0}$ (a further slight abuse of notation), and  \propref{prop: E then D_f}
 gives  $D_{\overline{g} \circ E}$ is contiguity equivalent to a further trivial extension of $\overline{g}$, which is a trivial extension of $g$. It follows that $[g] = 0 \in F(X, x_0)$.  
\end{proof}

\begin{example}
Let $X$ be the octahedral model of $S^2$ from \exref{ex: sphere S2}.  Then we have $|X| = S^2$ and \thmref{thm: Face group iso homotopy group} implies an isomorphism of groups 
$F(X, x_0) \cong \pi_2(S^2) \cong \Z$.  We assert that the map $f:  \left(I_5 \times I_4, \partial (I_5 \times I_4) \right)  \to (X, x_0)$ illustrated in \figref{fig: simplicial map} represent the trivial element in $F(X, x_0)$.  (Briefly, this must be the case since $f$ omits the vertex $-\mathbf{e}_2$ from its image.) We further assert that the map represented array style in \figref{fig: Generator F(X)} below represents a generator of $F(X, x_0)$.    (Briefly, this is so as the spatial realization of this map corresponds to a generator of $\pi_2(S^2)$ when represented as a map $(I^2, \partial I^2) \to (S^2, x_0)$.)    

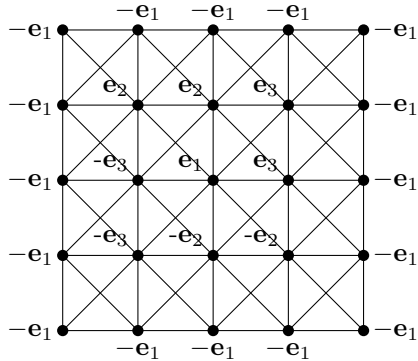
\begin{figure}
\[
\begin{tikzpicture}[scale=1]
\foreach \x in {0,...,4} {
 \foreach \y in {0,...,4} {
   \node[inner sep=1.5pt, circle, fill=black]  at (\x,\y) {};
  }
 }
\draw (0,0) grid (4,4);
\foreach \i in {0,..., 3} {
  \foreach \j in {0,..., 3} {
    \draw (\i,\j) -- (\i+1,\j+1);
  }
}
\foreach \i in {0,..., 3} {
  \foreach \j in {1,..., 4} {
    \draw (\i,\j) -- (\i+1,\j-1);
  }
}
\foreach \j in {0,..., 4} {
\node[left]  at (0,\j) {$-\mathbf{e}_1$};
\node[right]  at (4,\j) {$-\mathbf{e}_1$};
}
\foreach \i in {1,..., 3} {
\node[below]  at (\i,0) {$-\mathbf{e}_1$};
\node[above]  at (\i,4) {$-\mathbf{e}_1$};
}
\node[above left]  at (1,3) {$\mathbf{e}_2$};
\node[above left]  at (2,3) {$\mathbf{e}_2$};
\node[above left]  at (3,3) {$\mathbf{e}_3$};
\node[above left]  at (3,2) {$\mathbf{e}_3$};
\node[above left]  at (3,1) {-$\mathbf{e}_2$};
\node[above left]  at (2,1) {-$\mathbf{e}_2$};
\node[above left]  at (1,1) {-$\mathbf{e}_3$};
\node[above left]  at (1,2) {-$\mathbf{e}_3$};
\node[above left]  at (2,2) {$\mathbf{e}_1$};
\end{tikzpicture}
\]
\caption{A generator of the face group $F(X, x_0)$.} \label{fig: Generator F(X)}
\end{figure}

\end{example}

\section{Future Work}\label{sec: future work}

We have focussed exclusively on the second homotopy group in this paper but it seems reasonable to treat higher homotopy groups of all dimensions in a way similar to our development here. On the combinatorial side, defining a combinatorial version of $\pi_k(X)$, for $k \geq 3$,  would involve extension-contiguity equivalence classes of maps 
$$I_{n_1} \times \cdots \times I_{n_k} \to X$$
that send the boundary to the basepoint $x_0 \in X$. One would need to adapt and extend the material on simplicial approximation to this higher-dimensional context.  

We have an immediate application of our face group to our work in digital topology.  In previous work \cite{LOS19c, LuSc22}, we have defined a digital fundamental group and shown that our digital fundamental group is isomorphic to the edge group of the clique complex (a simplicial complex associated to the digital image by considering the digital image as a graph). More recently, in \cite{L-M-S-S-T} we have defined a digital second homotopy group.  With our face group now in hand, it is straightforward to establish a suitable chain of isomorphisms as we did in \cite{LuSc22} and conclude the following.

\begin{introtheorem}
Let $X$ be a digital image and $\mathrm{cl}(X)$ be the clique complex of $X$, as in \cite{LuSc22}. We have isomorphisms of groups
$$\pi_2(X) \cong F(\mathrm{cl}(X)) \cong \pi_2(|\mathrm{cl}(X)|),$$
with the $\pi_2(X)$ on the left denoting the digital second homotopy group of \cite{LuSc22} and the $\pi_2(|\mathrm{cl}(X)|)$ on the right denoting the ordinary (topological) second homotopy group of the spatial realization of $\mathrm{cl}(X)$.  
\end{introtheorem}

We will establish this result in a forthcoming paper.  Just as \cite{LuSc22} led to many results about the digital fundamental group, so too this result will greatly clarify our understanding of the digital second homotopy group.

A third direction for development involves the simplicial loop space of \cite{LuSc25}. This work stems from the idea that a face sphere may be viewed as a sequence of edge loops by reading across the rows of the face sphere in order.  Namely, a face sphere corresponds in this way to an ``edge loop of edge loops" in some suitable, larger simplicial complex. In \cite{LuSc25} we implement the program suggested by this idea.  We define a simplicial complex $\Omega X$ for any simplicial complex $X$ that serves as a simplicial version of the ordinary (topological) based loop space.  In particular, we show there an isomorphism of groups
$$E(\Omega X) \cong F(X),$$
between the edge group of the simplicial complex $\Omega X$ and the face group of the original simplicial complex $X$. This isomorphism is a combinatorial version of the isomorphism $\pi_1(\Omega X) \cong \pi_2(X)$ for a topological space $X$.


\begin{thebibliography}{10}

\bibitem{Bo99}
L.~Boxer, \emph{A classical construction for the digital fundamental group}, J.
  Math. Imaging Vision \textbf{10} (1999), no.~1, 51--62. \MR{1692842}

\bibitem{Bre}
Glen~E. Bredon, \emph{Topology and geometry}, Graduate Texts in Mathematics,
  vol. 139, Springer-Verlag, New York, 1997, Corrected third printing of the
  1993 original. \MR{1700700}

\bibitem{Gr02}
Marco Grandis, \emph{An intrinsic homotopy theory for simplicial complexes,
  with applications to image analysis}, Appl. Categ. Structures \textbf{10}
  (2002), no.~2, 99--155. \MR{1891107}

\bibitem{HiWi60}
P.~J. Hilton and S.~Wylie, \emph{Homology theory: {A}n introduction to
  algebraic topology}, Cambridge University Press, New York, 1960. \MR{115161}

\bibitem{Koz08}
Dmitry Kozlov, \emph{Combinatorial algebraic topology}, Algorithms and
  Computation in Mathematics, vol.~21, Springer, Berlin, 2008. \MR{2361455}

\bibitem{LOS19c}
G.~Lupton, J.~Oprea, and N.~Scoville, \emph{A fundamental group for digital
  images}, J. Appl. Comput. Topol. \textbf{5} (2021), no.~2, 249--311.
  \MR{4259435}

\bibitem{LuSc25}
G.~Lupton and J.~Scott, \emph{A simplicial loop space}, arXiv:2504.11223
  [math.AT], 2025.

\bibitem{LuSc22}
G.~Lupton and N.~Scoville, \emph{Digital fundamental groups and edge groups of
  clique complexes}, J. Appl. Comput. Topol. \textbf{6} (2022), no.~4,
  529--558. \MR{4496690}

\bibitem{L-M-S-S-T}
Gregory Lupton, Oleg Musin, Nicholas~A. Scoville, P.~Christopher Staecker, and
  Jonathan Trevi\~no Marroqu\'in, \emph{A second homotopy group for digital
  images}, J. Algebraic Combin. \textbf{60} (2024), no.~3, 781--815.
  \MR{4815675}

\bibitem{Mau96}
C.~R.~F. Maunder, \emph{Algebraic topology}, Dover Publications, Inc., Mineola,
  NY, 1996, Reprint of the 1980 edition. \MR{1402473}

\end{thebibliography}


\providecommand{\bysame}{\leavevmode\hbox to3em{\hrulefill}\thinspace}
\providecommand{\MR}{\relax\ifhmode\unskip\space\fi MR }
\providecommand{\MRhref}[2]{%
  \href{http://www.ams.org/mathscinet-getitem?mr=#1}{#2}
}
\providecommand{\href}[2]{#2}

\end{document}